\documentclass[11pt]{article}
\usepackage[margin=0.86in]{geometry}
\usepackage{amsmath,amssymb,amsfonts,amsthm}
\usepackage{graphicx,booktabs,array,tikz,needspace}
\usepackage[numbers]{natbib}
\usepackage[colorlinks=true,linkcolor=blue,citecolor=green,urlcolor=black,
            hypertexnames=false]{hyperref}
\usepackage{microtype}
\hypersetup{
  pdftitle={Exact Simulation of Multivariate Diffusions and Bridges},
  pdfauthor={Jose Blanchet, Jikai Jin, Hao Liu},
  pdfkeywords={exact simulation, diffusion bridge, stochastic differential equations, Monte Carlo method, complexity}
}

\newcounter{algorithm}
\newenvironment{algorithm}[1][]{\par\addvspace{\medskipamount}\noindent\begin{minipage}{\linewidth}\refstepcounter{algorithm}}{\end{minipage}\par\addvspace{\medskipamount}}
\newcommand{\algcaption}[1]{\noindent\textbf{Algorithm \thealgorithm. #1}\par\nopagebreak[4]\smallskip\nopagebreak[4]}
\newcommand{\alginout}[2]{\noindent\textit{Input:} #1\par\nopagebreak[4]\noindent\textit{Output:} #2\par\nopagebreak[4]\smallskip\nopagebreak[4]}
\newcommand{\E}{\mathbb{E}}
\newcommand{\Pp}{\mathbb{P}}
\newcommand{\R}{\mathbb{R}}
\newcommand{\Id}{I_d}
\newcommand{\dd}{\,\mathrm{d}}
\newcommand{\ind}{\mathbf{1}}
\newcommand{\calN}{\mathcal{N}}
\newcommand{\calL}{\mathcal{L}}

\newcommand{\Wk}{\mathcal{W}}
\newcommand{\Mc}{\Gamma_{\mathrm{coef}}}
\newcommand{\TC}{\operatorname{TransitionCoin}}
\theoremstyle{plain}
\newtheorem{theorem}{Theorem}[section]
\newtheorem{proposition}[theorem]{Proposition}
\newtheorem{lemma}[theorem]{Lemma}

\theoremstyle{definition}
\newtheorem{definition}[theorem]{Definition}
\newtheorem{assumption}[theorem]{Assumption}
\newtheorem{remark}[theorem]{Remark}

\title{Exact Simulation of Multivariate Diffusions and Bridges}

\author{%
Jose Blanchet%
\thanks{The authors are listed in alphabetical order.}%
\textsuperscript{\phantom{*},}%
\thanks{Department of Management Science and Engineering,
Stanford University. Email: \texttt{jose.blanchet@stanford.edu}}
\and
Jikai Jin%
\footnotemark[1]\textsuperscript{\phantom{*},}%
\thanks{Institute for Computational and Mathematical Engineering,
Stanford University. Email: \texttt{jkjin@stanford.edu}}
\and
Hao Liu%
\footnotemark[1]\textsuperscript{\phantom{*},}%
\thanks{Department of Management Science and Engineering,
Stanford University. Email: \texttt{haoliu20@stanford.edu}}
}
\date{\today}

\begin{document}
\maketitle

\begin{abstract}
We provide the first generic exact sampling
algorithms with finite expected total cost for uniformly elliptic
multivariate diffusions and the corresponding diffusion bridges with bounded coefficients and bounded continuous first derivatives. We formulate a computational cost model that counts work in terms of generation of Gaussian and uniform random variables to execute acceptance--rejection sampling, as well as function evaluations of the drift and diffusion coefficients, together with elementary operations such as sums and multiplications. Our constructions combine randomized
parametrix expansions with a finite weighted graph that
incorporates both bridge endpoints into the proposal. The exact sampling of the diffusion
is minimax optimal jointly in the horizon and number of requested
observations, while bridge sampling is minimax optimal in the
number of observations for fixed horizon and endpoints.
\end{abstract}

\medskip
\noindent\textbf{MSC2020 subject classifications.}
Primary 65C05; secondary 60H10, 60J60, 65C30, 68Q25.


\tableofcontents

\section{Introduction}\label{sec:introduction}

Exact simulation of a diffusion seeks to generate the exact sample paths of the process. In this paper, we consider the generation of the values of the process at prescribed
times with their joint distribution, without any bias. Exact samples yield Monte Carlo estimates without time-discretization
bias. They also preserve useful structure: nonnegative functions
yield nonnegative estimates, and functions convex in a decision
variable for every state yield convex sample-average
objectives.
Beskos and Roberts~\citep{BeskosRoberts2005} showed how to generate exact samples from one-dimensional diffusions using acceptance--rejection sampling. The key idea is to reduce the sampling problem to the case of diffusions with constant volatility / diffusion coefficient. This, however, can only be done in generality in the one-dimensional case using the so-called Lamperti transformation. The key idea in the Beskos-Roberts algorithm is to propose samples from the Brownian measure and apply It\^o's lemma to transform the stochastic integral that appears in the Girsanov transformation into a Riemann integral. Then, under suitable boundedness assumptions on the drift and its first and second derivatives, the acceptance step can be interpreted in terms of Poisson thinning. Subsequently, the boundedness assumptions were removed by localization and related techniques in the work of retrospective simulation
\citep{BeskosPapaspiliopoulosRoberts2006}, layered Brownian bridges
\citep{BeskosPapaspiliopoulosRoberts2008} and localization
\citep{ChenHuang2013}. 

However, what remains key in these techniques are two essential features. First, the ability to reduce the exact sampling problem to one involving a diffusion with a constant volatility matrix, and second being able to represent the Girsanov likelihood ratio in terms of a Riemann integral. The first feature creates substantial limitations in multivariate
settings~\citep{AitSahalia2008}, and even if the target diffusion
itself has constant volatility, the second feature further restricts
the drift to be a gradient vector
field~\citep{BeskosPapaspiliopoulosRoberts2008}.

The work of Blanchet and Zhang~\citep{BlanchetZhang2020} developed a
different approach that samples multivariate diffusion skeletons without
these restrictions, combining tolerance-enforced simulation \citep{BlanchetChenDong2017}, rough-path
estimates and randomized multilevel ideas. Their algorithms terminate
almost surely but have infinite expected computational cost. Thus the existence
of an exact sampler leaves a further question about the computational
resources needed to use it. 

Given these limitations, it is natural to ask: \emph{can exact simulation of multivariate diffusions with general time- and state-dependent coefficients be achieved with finite expected total computational cost?}

One of the main contributions of this paper is to formulate a \emph{query-only} computational model for exact
diffusion simulation and construct the first generic exact samplers with finite expected
total work. The algorithms know the coefficient class bounds and
access the coefficients only through scalar point evaluations. By \emph{query-only} we mean that we count all operations, including preprocessing,
coefficient queries, auxiliary randomization, rejected proposals
and output --- but we do not assume we have access to the transition density. We consider uniformly elliptic multivariate diffusions
with globally bounded, time- and state-dependent coefficients and
bounded continuous derivatives through a fixed order $k\ge1$.

Further, some applications of diffusion sampling, in particular
those involving parametric inference of diffusion models, motivate
the study of a multivariate diffusion conditioned on its terminal
value at a particular time. These processes are known as diffusion
bridges. For exact and other bridge-based approaches to inference
for discretely observed diffusions, see Beskos, Papaspiliopoulos,
Roberts and
Fearnhead~\citep{BeskosPapaspiliopoulosRobertsFearnhead2006},
Delyon and Hu~\citep{DelyonHu2006}, Bladt and
S{\o}rensen~\citep{BladtSorensen2014}, and Schauer, van der Meulen
and van Zanten~\citep{SchauerVanderMeulenVanZanten2017}. Bladt, Finch and S{\o}rensen~\citep{BladtFinchSorensen2016}
develop coupling-based proposals and MCMC methods for
time-homogeneous ergodic diffusion bridges. Their MCMC
construction targets the bridge law as its invariant distribution.
The present work allows time-dependent coefficients and does not
require ergodicity, under the boundedness and regularity assumptions
specified below. It also requires exactness of the sample returned
at termination and accounts for the expected total work needed
to obtain it.

Sampling from multivariate diffusion bridges with general
time- and state-dependent coefficients is challenging because they are defined on a conditioning event having null probability. Therefore, the exact sampling algorithms discussed earlier cannot be directly used for exact bridge diffusion sampling --- unless we are once again in the setting of constant volatility and gradient drift vector fields; in those cases, exact bridge diffusion sampling can be adapted from the exact (unconditional) sampling techniques discussed earlier. So, given this further application, it is also natural to ask: \emph{Can we perform exact simulation of multivariate diffusion bridges with general time- and state-dependent coefficients within finite expected total computational cost?}

The second main contribution of this paper is to answer this question in the affirmative. 

Moreover, we further prove that the unconditional exact sampling construction is minimax optimal jointly in the terminal sampling time horizon $T$ and the number of sampled observations $m$, while the bridge construction is minimax optimal in $m$ for fixed horizon and endpoints. By minimax optimality, we mean achieving, up to constant factors, the smallest worst-case expected total work among all exact
algorithms in the same computational model. The worst case is
taken over all diffusion coefficients satisfying the prescribed
coefficient, derivative and ellipticity bounds, and all admissible
grids of $m$ observation times.

Linear horizon bounds for computational cost already appear in
exact simulation based on successive short intervals
\citep{BeskosRoberts2005,BeskosPapaspiliopoulosRoberts2006}.
Our unconditional sampling result establishes linear total expected work for the present multivariate coefficient class in a query-only model,
together with a matching information lower bound. The minimax
analysis is related to information-based complexity
\citep{TraubWasilkowskiWozniakowski1988} and optimal approximation
of stochastic differential equations
\citep{HofmannMullerGronbachRitter2000,MullerGronbach2002}.

The unconditional exact sampling construction follows a natural strategy: simulate
successive short transitions, subdividing any requested time gap that
exceeds a permitted duration, computed from the assumed bounds on the process coefficients. Each short transition is sampled by
Gaussian rejection. Although the transition density is unavailable,
the acceptance decision can be implemented by a so-called \emph{transition
coin}, a binary random variable whose mean is the density divided by
a known Gaussian envelope. Constructing this coin with uniformly
bounded expected work makes the short-transition strategy executable
and gives the stated total-work bound.

For bridges, the same short-transition machinery suggests generating
an exact diffusion prefix and accepting it according to its final
transition density to the prescribed endpoint. This construction
provides a natural way to impose the terminal condition, but its
expected work can grow exponentially with the horizon. Yet another important contribution of this paper is to give a construction that controls this conditioning cost by incorporating both endpoints into the proposal through a finite weighted graph of spatial cells and time layers. A graph path
selects a sequence of cells, within which continuous states are
proposed and accepted by rejection sampling. Additional proposal
components cover large increments and spatial exits, preserving
the full bridge law. For fixed endpoints and a fixed number of
requested observations, the expected work grows at most
polynomially as the horizon increases. Related uses of graphs to represent and sample stochastic
trajectories include stochastic roadmap simulation for molecular
motion \citep{ApaydinEtAl2003} and transition-path sampling on
Markovian networks \citep{SharpeWales2020}. 

Our coin constructions use randomized parametrix expansions to
obtain unbiased estimators of density ratios, together with pathwise
bounds that turn these estimators into valid acceptance probabilities.
This approach is related to probabilistic parametrix representations
and unbiased diffusion estimation
\citep{BallyKohatsuHiga2015,AnderssonKohatsuHiga2017,HenryLabordereTanTouzi2017},
as well as randomized debiasing
\citep{McLeish2011,RheeGlynn2015}. The bridge construction also uses
Bernoulli factories, including the linear factory of
Huber~\citep{Huber2016}, and exact sampling from unknown weights.
Related methods for the latter problem appear in Dughmi, Hartline,
Kleinberg and Niazadeh~\citep{DughmiHartlineKleinbergNiazadeh2017}.

Unbiased estimation provides another way to eliminate
time-discretization bias. Randomized debiasing methods construct
estimators of specified expectations
\citep{McLeish2011,RheeGlynn2015,BlanchetGlynn2015}, without necessarily generating
samples from the target distribution.
Glynn~\citep{Glynn2016ExactEstimation} discusses unbiased estimation
as a useful relaxation of exact simulation, particularly when
sampling from the target law is difficult. The distinction also
matters for applications: unbiasedness alone does not guarantee
the structural properties of direct sample averages described
above, as discussed by Blanchet and
Zhang~\citep{BlanchetZhang2020}. For a further perspective on exact
simulation of diffusions and its open computational questions,
see Blanchet~\citep{Blanchet2022OpenProblems}.

The rest of the paper is organized as follows. Section~\ref{sec:setup} specifies the coefficient class, sampling
targets and computational model. Sections~\ref{sec:forward}
and~\ref{poly:section} present the unconditional diffusion and the bridge constructions,
respectively. Section~\ref{sec:minimax} gives the minimax guarantees,
and Section~\ref{sec:discussion} concludes and discusses further questions.
Appendix~\ref{app:accounting} records the work details, and the remaining appendices give the detailed
constructions and proofs.

\section{Sampling problems and computational model}\label{sec:setup}\label{cx:section}

We first specify the coefficient
class and these two target laws, then describe the information available
to the algorithm and the work charged for using it.

\subsection{The diffusion and its coefficients}

Let $X$ solve the $d$-dimensional stochastic differential equation
\begin{equation*}
 \dd X_t=b(t,X_t)\dd t+\sigma(t,X_t)\dd W_t,\qquad X_0=x,
 \quad a=\sigma\sigma^\top,
\end{equation*}
where $W$ is a standard $d$-dimensional Brownian motion. For vectors,
$|\cdot|$ denotes the Euclidean norm. For matrices,
$P:Q=\operatorname{tr}(P^\top Q)$ and $P\preceq Q$ means that $Q-P$
is positive semidefinite; $\Id$ is the $d\times d$ identity.
The space $C_b^k$ consists of functions with bounded continuous partial
derivatives through total order $k$, including order zero.

The constructions below compare short diffusion transitions with Gaussian
kernels. We impose uniform bounds on the coefficients and their
derivatives, together with uniform ellipticity. We use the following assumption
throughout the paper.

\begin{assumption}\label{ass:class}
The integers $d,k\ge1$, the bounds $\Mc>0$ and $0<\lambda\le\Lambda$,
and the horizon $T>0$ are known to the algorithm. Every component of
$b$ and $\sigma$ belongs to $C_b^k([0,T]\times\R^d)$, every time--space
derivative of total order at most $k$, including order zero, has absolute
value at most $\Mc$, and
\[
 \lambda\Id\preceq a(t,z)\preceq\Lambda\Id,
 \qquad (t,z)\in[0,T]\times\R^d.
\]
Derivatives at the time endpoints have continuous one-sided extensions.
\end{assumption}

For fixed $d,k,\Mc,\lambda,\Lambda$, let $\mathcal F_T$ denote the full
coefficient class satisfying Assumption~\ref{ass:class}; we assume this
class is nonempty. 
These assumptions imply existence and pathwise uniqueness of a
nonexplosive strong solution \citep{KaratzasShreve1991}. The diffusion
also has a strictly positive, jointly continuous transition density
$p(s,u;t,v)$ satisfying the Chapman--Kolmogorov identity pointwise.
We use this continuous density version to specify both sampling targets.

\subsection{The unconditional and bridge targets}

For $0=t_0<t_1<\cdots<t_m\le T$, the \emph{unconditional target} is the joint
law of $(X_{t_1},\ldots,X_{t_m})$ for the initial state $X_0=x$.
The Markov property expresses its density as the product of successive
transition densities:
\begin{equation}\label{eq:forwardtarget}
 \prod_{i=1}^m p(t_{i-1},z_{i-1};t_i,z_i),\qquad z_0=x.
\end{equation}
Thus an exact sampler for individual transitions can be applied
successively to produce the requested vector.

For bridge sampling, both $X_0=x$ and $X_T=y$ are fixed, and the requested
times satisfy $0=t_0<t_1<\cdots<t_m<t_{m+1}=T$. Conditioning the joint
density of the intermediate states and the terminal state gives the
\emph{bridge target}
\begin{equation}\label{eq:bridgetarget}
 \frac{\prod_{i=1}^{m+1}p(t_{i-1},z_{i-1};t_i,z_i)}{p(0,x;T,y)},
 \qquad z_0=x,\quad z_{m+1}=y.
\end{equation}

To describe the constructions, we call $e=(s,u;t,v)$, $s<t$, an
\emph{edge}: it records a transition between the indicated time--state
pairs. A \emph{coin} with success probability $p$ is a $\{0,1\}$-valued
random variable of mean $p$. For known $p$, such a coin is generated by
comparison with a fresh uniform random variable.

\subsection{Information and work}\label{sec:model}
For coefficients satisfying Assumption~\ref{ass:class}, the algorithm
receives the class bounds and accesses $b$ and $\sigma$ through point
queries to their scalar components. Queries may depend on previous
answers and random draws. Coefficient derivatives, formulas, integrals
and transition densities are not supplied. An algorithm whose coefficient
access is restricted to these queries is called \emph{query-only}.

To state a work guarantee, we must also account for the computation
between queries. We use an exact-real model with fixed positive finite
costs for scalar arithmetic and comparisons, square roots, exponentials,
logarithms, coefficient queries, independent scalar Gaussian and uniform
draws, memory operations and output writes. Known constants such as $\pi$
are available. The work $\Wk$ counts preprocessing, rejected proposals,
recursive calls and all auxiliary computation. Matrix calculations and
random-length loops are charged for their scalar operations.
Appendix~\ref{app:accounting} gives the detailed conventions, in the
spirit of real-number computation and nonuniform random generation
\citep{BlumShubSmale1989,Devroye1986}. Unless stated otherwise, constants in work bounds may depend on the fixed
class parameters and primitive costs.

The unconditional sampling construction in Section~\ref{sec:forward} develops an exact
short-transition sampler from coefficient values. Section~\ref{poly:section}
uses the same transition coins, together with proposals that incorporate
the prescribed terminal state, to sample bridges.

\section{Exact diffusion sampling}\label{sec:forward}\label{core:section}

Recall the unconditional target \eqref{eq:forwardtarget}. Its product form
suggests sampling successive transitions, subdividing a requested time
gap whenever it is too long. The main task is therefore to sample one
short transition using only coefficient values. We do this by Gaussian
rejection sampling. Although the transition density cannot be evaluated,
its acceptance probability can be implemented by a coin.

For $c,h>0$, write
\[
 G_{c,h}(w)=(2\pi ch)^{-d/2}\exp\{-|w|^2/(2ch)\}.
\]
We construct a computable duration $h_{\mathrm{env}}>0$ and a Gaussian
envelope $E_h=K_{\mathrm{rej}}G_{4\kappa/3,h}$, where $\kappa=2\Lambda$,
together with a routine $\TC(s,u;t,v)$ of mean
$p(s,u;t,v)/E_{t-s}(v-u)$. Its expected work is uniformly bounded for
$0<t-s\le h_{\mathrm{env}}$, including arbitrary endpoint displacements.
The construction starts with a randomized parametrix expansion, produces
a local coin, and then extends it by a midpoint recursion. Detailed
implementations and proofs are given in Appendix~\ref{app:forward-proofs}.

\subsection{Local randomization}\label{sec:parametrixexpansion}

For an edge $e=(s,u;t,v)$, let $h=t-s$, $w=v-u$ and $A=a(t,v)$.
Freezing the covariance at the terminal point gives the explicit kernel
\[
 g(e)=\frac{\exp\{-w^\top A^{-1}w/(2h)\}}
 {(2\pi h)^{d/2}\sqrt{\det A}}.
\]
Its defect under the backward diffusion operator
$\calL_s=b(s,u)\cdot\nabla+\tfrac12a(s,u):D^2$ is
\begin{equation}\label{core:defect}
 \begin{aligned}
 H(e)&=(\partial_s+\calL_s)g(e)
 =g(e)\left[b(s,u)\cdot\ell+
 \tfrac12\{a(s,u)-A\}:\left(\ell\ell^\top-A^{-1}/h\right)\right],
 \qquad \ell=A^{-1}w/h.
 \end{aligned}
\end{equation}
Thus both $g$ and $H$ are evaluated from coefficient values. For kernels
$F,H$, define their time--space convolution by
\[
 (F\otimes H)(s,u;t,v)=\int_s^t\int_{\R^d}
 F(s,u;r,z)H(r,z;t,v)\dd z\dd r.
\]
The parametrix representation is the absolutely convergent series
\begin{equation}\label{core:parametrix-series}
 p(e)=g(e)+\sum_{j\ge1}(g\otimes H^{\otimes j})(e).
\end{equation}
For $0<h\le1$, computable class constants satisfy
\[
 g(e)\le C_gG_{\kappa,h}(w),\qquad
 |H(e)|\le C_Hh^{-1/2}G_{\kappa,h}(w),\qquad
 R(e):=\frac{p(e)}{G_{\kappa,h}(w)}\le C_p^*.
\]
The proof of these bounds can be found in Section \ref{sec:B.1}.

To implement the series, we randomize the times and intermediate states
in one convolution term. For each $j\ge1$, we can generate $\Upsilon_j$ using
$O(j+1)$ work (details see Section \ref{sec:B.2}). With
$\mathsf Z_j(h)=\pi^{j/2}h^{j/2}/\Gamma(1+j/2)$ (where $\Gamma$ denotes the gamma function), it satisfies
\begin{equation}\label{core:Tnproperties}
 \E\Upsilon_j=\frac{(g\otimes H^{\otimes j})(e)}{G_{\kappa,h}(w)},
 \qquad
 |\Upsilon_j|\le C_gC_H^j\mathsf Z_j(h).
\end{equation}
The random times absorb the inverse square-root singularities in the
defect bounds, and a Gaussian bridge supplies the intermediate states.
The normalizers are computed by a recurrence, so no Gamma-function
primitive is required.

Draw $J\ge1$ with $\Pp(J=j)=2^{-j}$ and generate $\Upsilon_J$
conditionally on $J$. With $r_0=g(e)/G_{\kappa,h}(w)$, we define the local estimator $\widehat R_{\mathrm{loc}}=r_0+2^J\Upsilon_J.$
Therefore, $\E\widehat R_{\mathrm{loc}}=R(e)$. Its expected work is
uniformly bounded since $\E J=2$. The correction may have either
sign, but on short edges with nearby endpoints the leading term
dominates every realization. Specifically, for each $L\ge1$ there is
a computable $h(L)>0$ such that $h\le h(L)$ and $|w|\le L\sqrt h$
imply
\begin{equation}\label{core:localbounds}
 \tfrac34r_0\le\widehat R_{\mathrm{loc}}\le\tfrac54r_0
 <B_{\mathrm{loc}}:=\tfrac32 C_g
 \quad\text{almost surely}.
\end{equation}
Consequently, an independent uniform comparison with a normalized
version of this estimator is a valid coin. Lemma~\ref{core:localpositive}
gives the explicit choice of $h(L)$.

\subsection{A coin for arbitrary endpoints}\label{core:spatialsection}

The local construction does not cover all states produced by a Gaussian
proposal. To extend this to arbitrary endpoints, we denote $\zeta=w/\sqrt h$ and $\bar s=(s+t)/2$.
Under the Gaussian midpoint law
$Z=(u+v)/2+\sqrt{\kappa h/4}\,\xi$, $\xi\sim\calN(0,\Id)$,
the Chapman--Kolmogorov identity gives
\begin{equation}\label{core:midpointidentity}
 R(e)=\E[R(e_1)R(e_2)],\qquad
 e_1=(s,u;\bar s,Z),\quad e_2=(\bar s,Z;t,v).
\end{equation}
The normalized displacements of the two new edges are
\[
 \zeta_1=\frac{\zeta}{\sqrt2}+\sqrt{\kappa/2}\,\xi,
 \qquad
 \zeta_2=\frac{\zeta}{\sqrt2}-\sqrt{\kappa/2}\,\xi.
\]
We choose an envelope of the form
\begin{equation}\label{core:spatial-envelope}
 C(\zeta)=b_{\mathrm{env}}+k_{\mathrm{env}}e^{a_0|\zeta|^2},
 \qquad a_0=(8\kappa)^{-1},
 \qquad S(\zeta)=\E[C(\zeta_1)C(\zeta_2)].
\end{equation}
Both $C$ and $S$ are explicitly computable. The constants
$b_{\mathrm{env}},k_{\mathrm{env}},L_{\mathrm{env}}$ can be chosen
from the class bounds so that, with
$h_{\mathrm{env}}=h(L_{\mathrm{env}})$,
\[
 R(e)\le C(\zeta),\qquad
 0\le\widehat R_{\mathrm{loc}}\le C(\zeta)
 \ \text{if }|\zeta|\le L_{\mathrm{env}},\qquad
 \frac{S(\zeta)}{C(\zeta)}\le\frac14
 \ \text{if }|\zeta|>L_{\mathrm{env}}.
\]
These inequalities hold for every $0<h\le h_{\mathrm{env}}$;
their calibration is given in Section~\ref{supp:spatialproof}.
This construction leads to the following recursive coin of mean $R(e)/C(\zeta)$.

\begin{algorithm}
\algcaption{$\operatorname{SpatialCoin}(s,u;t,v)$}
\label{core:spatialalgorithm}
\alginout{$e=(s,u;t,v)$ with $0<t-s\le h_{\mathrm{env}}$.}
{a coin of mean $R(e)/C((v-u)/\sqrt{t-s})$.}
All draws use fresh randomness conditional on the current inputs.
\begin{enumerate}
\item Set $h=t-s$ and $\zeta=(v-u)/\sqrt h$.
\item If $|\zeta|\le L_{\mathrm{env}}$, generate
$\widehat R_{\mathrm{loc}}$ and an independent uniform random variable $U$.
Return $\ind\{U\le\widehat R_{\mathrm{loc}}/C(\zeta)\}$.
\item Otherwise, draw a uniform random variable $U_0$. If
$U_0>S(\zeta)/C(\zeta)$, return zero.
\item Draw $\xi$ with density
\[
 \widetilde f_\zeta(\xi)=
 \frac{C(\zeta_1)C(\zeta_2)}{S(\zeta)}
 (2\pi)^{-d/2}e^{-|\xi|^2/2},
\]
and set $\bar s=(s+t)/2$ and
$Z=(u+v)/2+\sqrt{\kappa h/4}\,\xi$.
\item Call $\operatorname{SpatialCoin}(s,u;\bar s,Z)$ and
$\operatorname{SpatialCoin}(\bar s,Z;t,v)$ independently conditional
on $Z$, and return the product of their outputs.
\end{enumerate}
\end{algorithm}

The tilted midpoint density $\widetilde f_\zeta$ is a four-component Gaussian mixture (see Section \ref{supp:spatialproof}).
Its weighting cancels the two envelope denominators in the recursive
coin probabilities, recovering \eqref{core:midpointidentity}. A call
makes two further calls with probability at most $1/4$, so the expected
total number of calls is at most two. Section \ref{supp:spatialproof}
establishes the claimed mean, almost-sure termination and uniformly
bounded expected work.

\subsection{Short transitions and the requested skeleton}
\label{core:shorttransitionalgorithm}

To obtain a Gaussian rejection envelope, set
$A_0=b_{\mathrm{env}}+k_{\mathrm{env}}$. Since
$C(\zeta)\le A_0e^{a_0|\zeta|^2}$, define
\begin{equation}\label{core:gaussianenvelope}
 E_h(w)=A_0G_{\kappa,h}(w)e^{a_0|w|^2/h}
       =K_{\mathrm{rej}}G_{4\kappa/3,h}(w),
 \qquad K_{\mathrm{rej}}=A_0(4/3)^{d/2}.
\end{equation}
Generate $B_{\mathrm{sp}}=\operatorname{SpatialCoin}(s,u;t,v)$
and an independent $U\sim\operatorname{Uniform}(0,1)$, and define another coin via
\begin{equation*}
 \TC(e)=B_{\mathrm{sp}}
 \ind\left\{U\le\frac{C(\zeta)}{A_0e^{a_0|\zeta|^2}}\right\}.
\end{equation*}
The additional comparison adjusts the normalization, giving
\begin{equation}\label{core:shortcoinlaw}
 \Pp\{\TC(e)=1\}=\frac{p(e)}{E_h(v-u)}.
\end{equation}
The extra work is bounded, so the transition coin retains the uniform
expected-work guarantee; see the details in Lemma~\ref{core:shortcoin}.

The \emph{short-transition algorithm} from $(s,u)$ to time $t$, with
$0<h=t-s\le h_{\mathrm{env}}$, now has a simple form: propose
$V\sim\calN(u,(4\kappa/3)h\Id)$ and accept if a fresh
$\TC(s,u;t,V)$ returns one. On failure, repeat the entire trial
with fresh randomness. The proposal density multiplied by the
acceptance probability is $p(s,u;t,v)/K_{\mathrm{rej}}$. Hence a
trial succeeds with probability $1/K_{\mathrm{rej}}$, and its
accepted state has the required transition law. This short-transition algorithm has
uniformly bounded expected work by Lemma~\ref{core:shorttransition}.

For the requested grid, subdivide each gap into
\begin{equation}\label{fw:counts}
 n_i=\left\lceil(t_i-t_{i-1})/h_{\mathrm{env}}\right\rceil
 \quad\text{equal steps},\qquad
 N_{\mathrm{short}}=\sum_{i=1}^m n_i
 \le m+t_m/h_{\mathrm{env}}.
\end{equation}
Then every resulting duration is admissible. An accepted state is retained
while the next transition is sampled, so a rejection repeats only the
current short-transition trial. Now we are ready to present the exact diffusion sampling algorithm.

\begin{algorithm}
\algcaption{Unconditional diffusion skeleton}\label{fw:algorithm}
\alginout{an initial state $x$ and times $0=t_0<t_1<\cdots<t_m\le T$.}
{a vector with density \eqref{eq:forwardtarget}.}
\begin{enumerate}
\item Calibrate the short-transition constants from the class bounds
and set the current time and state to $(0,x)$.
\item For $i=1,\ldots,m$, subdivide $[t_{i-1},t_i]$ as in
\eqref{fw:counts}. On each step, apply the short-transition routine
from the current time and state, then replace them by the step's
terminal time and accepted state. At $t_i$, store the state as the
$i$th output. Discard auxiliary states once their successors have
been accepted.
\item Return the stored vector.
\end{enumerate}
\end{algorithm}

Successive exact transitions give the unconditional target, while
\eqref{fw:counts} bounds their number without any restriction on
observation spacings. The resulting guarantee counts all operations,
including calibration and output. Its proof can be found in
Appendix~\ref{app:forward-proofs}.

\begin{theorem}[Unconditional sampling with linear expected work]\label{fw:upper}
Algorithm~\ref{fw:algorithm}, denoted by $A_F$, is query-only,
terminates almost surely and samples \eqref{eq:forwardtarget} exactly.
There is a constant $C>0$, depending only on the fixed coefficient
class and primitive costs, such that for every $T>0$, $x\in\R^d$,
$m\ge1$ and the grid,
\begin{equation}\label{cx:forward-upper}
 \sup_{(b,\sigma)\in\mathcal F_T}\E\Wk_{A_F}\le C(t_m+m).
\end{equation}
The constant is uniform in the initial state, horizon and requested grid.
\end{theorem}

\section{Exact bridge sampling}\label{poly:section}

Recall that the bridge target \eqref{eq:bridgetarget} fixes both
endpoints. A natural extension of unconditional sampling generates an exact
prefix and accepts it according to its final transition density to the
prescribed endpoint. Its expected work can grow exponentially with the
horizon, even for constant coefficients and equal endpoints; see
Section~\ref{br5:terminal-rejection}. 

We therefore construct a proposal
that incorporates both endpoints. It first samples a bridge on an
auxiliary time grid, then inserts the requested observations from their
conditional bridge laws.
The auxiliary-grid construction uses a finite graph of spatial cells.
Graph paths represent skeletons with controlled increments that remain
in a bounded region. Two further types of proposal cover the remaining
skeletons, according to whether the first violation is a large increment
or an exit from this region. Each component has an envelope mass and an
acceptance test. Selecting components in proportion to these masses
makes their accepted contributions combine into the full bridge density.

We describe the construction on $[0,T_0]$, with $T_0=2n$, integer
$n\ge1$, and both endpoints zero. We include $n$ as an auxiliary
requested time if necessary. Section~\ref{bridge:general-construction}
reduces arbitrary horizons and endpoints to these inputs and discards
any additional output.

The construction uses the transition coin and short-transition sampler
of Section~\ref{core:section}. For convenience, write
$K_{\mathrm b}=K_{\mathrm{rej}}$ and $c=4\kappa/3$. For
$0<t-s\le h_{\mathrm{env}}$, the transition coin satisfies
\[
 \Pp\{\TC(s,u;t,v)=1\}
 =\frac{p(s,u;t,v)}{K_{\mathrm b}G_{c,t-s}(v-u)},
\]
and both routines have uniformly bounded expected work. All draws below
use fresh randomness conditional on their inputs, except for explicitly
shared estimates.

\subsection{The auxiliary grid and proposal components}\label{br5:proposal}

For an even integer $N$ and $h=T_0/N$, the unnormalized bridge density
on the equally spaced grid is
\[
 F(z_1,\ldots,z_{N-1})
 =\prod_{i=1}^N p((i-1)h,z_{i-1};ih,z_i),\qquad z_0=z_N=0.
\]
Its normalizing constant is $p_{T_0}=p(0,0;T_0,0)$, so the target
density is $F/p_{T_0}$. We first define the graph and the partition of
skeleton space used by the proposal.

\subsubsection{Time and space grids}
We first specify the time grid and spatial cells.
Fix computable class constants $D,R_0\ge1$, $h_0,c_1>0$ and
$0<c_0\le1$.
Let $\alpha\in(0,1]$ be the dyadic exponent in the relative density
estimate. The relative density estimate is established in
Section~\ref{br5:local-inputs-proof}, and the constants are calibrated
in Section~\ref{br5:parameter-choice}. Set
\begin{equation}\label{br5:grid}
 N=2\lceil n^4/h_0\rceil,\qquad h=T_0/N,\qquad
 L=\sqrt{Dnh},\qquad \rho=\frac{c_0\sqrt h}{1+\sqrt n},
 \qquad \ell=c_1\rho N^{-1/\alpha}.
\end{equation}
The coarse grid contains $n$ and $h\le h_0n^{-3}$. The cell width
$\ell$ makes the relative variation of the transition density on
each good edge, defined below, of order $1/N$. The product of $N$
cell envelopes therefore introduces only a bounded factor.

Partition space into half-open cubes of side $\ell$, called cells,
centered at $\ell z$, $z\in\mathbb Z^d$. The fixed endpoints are
singleton cells with center zero. For a transition edge $e=(s,u;t,v)$, let $u_0$ and $v_0$ be the
centers of the cells containing $u$ and $v$, respectively.
Define $e^0=(s,u_0;t,v_0)$, the transition between these centers
at the same times. We call $e$ good if $|v_0-u_0|\le L$,
and bad otherwise. For transitions between consecutive coarse times, the parameter
choices ensure, on every good edge,
\begin{equation}\label{br5:cell-errors}
 \left|\frac{p(e)}{p(e^0)}-1\right|\le\frac1{100N},\qquad
 \left|\frac{g(e)}{g(e^0)}-1\right|\le\frac1{100N},
 \qquad \sqrt d\,\ell\le L/4.
\end{equation}
In particular, a bad edge has actual displacement greater than
$3L/4$. These comparisons are verified in
Section~\ref{br5:parameter-choice}. 

To obtain a finite graph, we retain only the cells in a bounded
spatial region. Let $R=R_0(n+1)$ and $J_{\mathrm{box}}$ be the largest integer
with $J_{\mathrm{box}}\ell\le R$. Retain the cells with indices in
$[-J_{\mathrm{box}},J_{\mathrm{box}}]^d$, whose union is
\begin{equation}\label{br5:box}
 \mathcal B=[-B,B)^d,\qquad B=(J_{\mathrm{box}}+1/2)\ell,
 \qquad R_-=R-\ell/2\le B.
\end{equation}
For each internal coarse time $ih$, $1\le i\le N-1$, and each
retained spatial cell $C$, introduce a vertex $(ih,C)$. This
vertex records that the state at time $ih$ lies in $C$. Thus,
each internal time has one vertex for every retained cell; these
vertices form the layer corresponding to that time.
Add a single source vertex at time $0$ and a single terminal
vertex at time $T_0$, each corresponding to the fixed state zero
and the singleton cell $\{0\}$. Connect a vertex at time $ih$
to a vertex at time $(i+1)h$ by a directed edge precisely when
the centers of their associated cells are at distance at most
$L$. 

A path from the source to the terminal vertex selects one cell
$C_i$ at each internal time $ih$. It therefore represents all
continuous-state skeletons $(z_1,\ldots,z_{N-1})$ satisfying
$z_i\in C_i$ for every $i$. Let $E_0$ be the set of skeletons
represented by all such paths. Equivalently, a skeleton belongs
to $E_0$ if every internal state lies in $\mathcal B$ and every
transition edge is good.

\subsubsection{Partition and component selection}\label{br5:component-plan}
The graph represents the event $E_0$. To classify a skeleton outside
$E_0$, inspect its edges in time order, testing for a bad edge before
testing whether its terminal state leaves $\mathcal B$. Let $B_i$ be
the event that the first violation is a bad edge at position $i$,
$1\le i\le N$. Let $X_\zeta$ be the event that the first violation
is an exit along a good edge through face
$\zeta=(j,s)\in\{1,\ldots,d\}\times\{-1,1\}$, where $j$ is the
coordinate and $s$ its exit direction. If several coordinates leave
the box at once, use the first such coordinate. The suffix after the
first violation is unrestricted. Thus $E_0$, the $N$ events $B_i$,
and the $2d$ events $X_\zeta$ partition skeleton space.

We construct one proposal and acceptance test for each component.
The graph component has an unknown envelope mass $Z_0$, while the
other components have known masses $D_i$ and $D_\zeta$. Their
one-trial accepted subdensities will be, respectively,
\begin{equation}\label{br5:component-identities}
 \frac{F\ind_{E_0}}{Z_0},\qquad
 \frac{F\ind_{B_i}}{D_i},\qquad
 \frac{F\ind_{X_\zeta}}{D_\zeta}.
\end{equation}
Here an accepted density describes the output together with the
event that the trial succeeds; its integral is the success probability.
The masses and tests are specified below.

These identities determine how the components must be combined. Let 
\begin{equation}\label{br5:total-envelope}
 M_{\mathrm{env}}=Z_0+\sum_{i=1}^N D_i+\sum_\zeta D_\zeta.
\end{equation}
Select the graph component with probability $Z_0/M_{\mathrm{env}}$,
or a remaining component with probability $D_i/M_{\mathrm{env}}$
or $D_\zeta/M_{\mathrm{env}}$, and run its trial once. The selection
probabilities cancel the denominators in
\eqref{br5:component-identities}. Since the events form a partition,
one complete trial has accepted density $F/M_{\mathrm{env}}$.
On success its density is therefore $F/p_{T_0}$. On failure we restart
component selection. Section~\ref{br5:selection} implements this
selection without evaluating $Z_0$ and bounds the number of retries.

\subsection{The component represented by the graph}\label{br5:cell-component}

For $E_0$, we first select a graph path and then draw the intermediate
states uniformly in its cells. We choose the path weights so that this
two-stage proposal is obtained by normalizing an envelope for $F$
on $E_0$. This requires both sampling paths with unknown
weights and implementing acceptance probabilities involving unknown
transition densities.

\subsubsection{Cell envelopes and local density coins}\label{br5:local}
We use $\mathfrak e$ for a directed graph edge and retain $e$ for a
transition between specified diffusion states. A graph edge
$\mathfrak e$ from $(s,C)$ to $(t,C')$ represents all transitions
$e=(s,u;t,v)$ with $u\in C$ and $v\in C'$. Its center transition is
$e^0=(s,u_0;t,v_0)$, where $u_0,v_0$ are the cell centers.
We begin with the weight of one graph edge $\mathfrak e$. By \eqref{br5:cell-errors}, $U_{\mathfrak e}=\left(1+\frac1{10N}\right)p(e^0)$
upper-bounds the transition density for every pair of states in these cells.
The bound is constant throughout that pair of cells. Put $v_{\mathfrak e}=\ell^d$
when $\mathfrak e$ enters an internal cell and $v_{\mathfrak e}=1$ when it enters the fixed
terminal vertex.
Assign the edge weight $v_{\mathfrak e}U_{\mathfrak e}$. For a graph path $\gamma$, the product $\omega_\gamma=\prod_{\mathfrak e\in\gamma}v_{\mathfrak e}U_{\mathfrak e}$
is the integral of $\prod_{\mathfrak e\in\gamma}U_{\mathfrak e}$ over all internal states
in the selected cells. Define $Z_0=\sum_\gamma\omega_\gamma$.
We seek the path distribution
\[
 \Pp(\Gamma=\gamma)=\frac{\omega_\gamma}{Z_0}.
\]
Indeed, the subsequent uniform draws have conditional density
$1/\prod_{\mathfrak e\in\gamma}v_{\mathfrak e}$, so the resulting skeleton density is
$Z_0^{-1}\prod_{\mathfrak e\in\gamma}U_{\mathfrak e}$ on $E_0$. 

The weights cannot be evaluated because $p(e^0)$ is unknown. We
therefore factor each edge weight into a known factor and a coin mean:
\begin{equation}\label{br5:cell-envelope}
 \begin{aligned}
 U_{\mathfrak e}&=\left(1+\frac1{10N}\right)p(e^0),\\
 k_{\mathfrak e}&=v_{\mathfrak e}\left(1+\frac1{10N}\right)2g(e^0),
 &q_{\mathfrak e}&=\frac{p(e^0)}{2g(e^0)},\qquad k_{\mathfrak e}q_{\mathfrak e}=v_{\mathfrak e}U_{\mathfrak e}.
 \end{aligned}
\end{equation}
The factor $k_{\mathfrak e}$ is computable. The weighted-path procedure below
requires coins for the $q_{\mathfrak e}$, with means bounded away from zero.
Normalizing by the frozen density $g$ gives this property through the
following local density coin.

For a fixed constant $D_{\sharp}>0$, the routine
$\operatorname{SharpCoin}_n(e)$ (defined in Section \ref{br5:local-inputs-proof}) applies to transition edges satisfying
\begin{equation}\label{br5:moderate}
 0<h=t-s\le h_0n^{-3},\qquad |v-u|^2/h\le D_{\sharp}n,
\end{equation}
and returns a coin of mean $p(e)/(2g(e))\in[1/4,3/4]$ with uniformly
bounded expected work. The construction randomizes the parametrix
expansion to obtain an unbiased estimator $Y_n(e)$ of $p(e)/g(e)$
with $1/2\le Y_n(e)\le3/2$ almost surely. An independent uniform
comparison with $Y_n(e)/2$ gives the coin. To obtain these tighter
bounds, the Gaussian reference reflects the covariance near the
terminal point, while a broader reference controls contributions from
outside that neighborhood. Both contributions are retained.
Section~\ref{br5:local-inputs-proof} gives the construction and proof.
The common calibration in Section~\ref{br5:parameter-choice} covers
every use of this routine below.

In particular, $\operatorname{SharpCoin}_n(e^0)$ supplies the required
$q_{\mathfrak e}$-coin. We can thus represent the desired path weights as
$\omega_\gamma=\prod_{\mathfrak e\in\gamma}k_{\mathfrak e}q_{\mathfrak e}$ without evaluating either
$q_{\mathfrak e}$ or $U_{\mathfrak e}$.

\subsubsection{Sampling paths with unknown weights}\label{br5:weighted}
We now describe how to use the known edge factors and the available
coins to sample a path with the required distribution. The same procedure will implement component selection in
Section~\ref{br5:selection}. Consider a finite layered graph with
$E$ edges, source $o$, terminal $z$, at least one complete path, and
exactly $N$ edges on every complete path (since the time requested is fixed in our exact simulation problem). Each graph edge $\mathfrak e$ has a known factor
$k_{\mathfrak e}>0$ and an available coin of mean $w_{\mathfrak e}\in[1/4,1]$. Delete edges
that lie on no complete path. The required path probability is $\frac{\prod_{\mathfrak e\in\gamma}k_{\mathfrak e}w_{\mathfrak e}}
      {\sum_{\gamma'}\prod_{\mathfrak e\in\gamma'}k_{\mathfrak e}w_{\mathfrak e}}.$
For the application to our cell graph, the weight is $w_{\mathfrak e}=q_{\mathfrak e}$.

The procedure first uses initial estimates to form a computable
path proposal, then applies an acceptance test to obtain the target
distribution. To construct this test, we work with logarithms:
products of edge weights become sums, and the ratio of target and
proposal weights becomes an exponential. Set $L_0=2\log16$ and $c_{\mathfrak e}=1+\frac{\log w_{\mathfrak e}}{L_0}$
so that
$k_{\mathfrak e}w_{\mathfrak e}
 =k_{\mathfrak e}\exp\{L_0(c_{\mathfrak e}-1)\}$.
A construction based on a linear Bernoulli factory gives
$c_{\mathfrak e}$-coins from $w_{\mathfrak e}$-coins without
evaluating the unknown logarithm.

At each edge, generate an initial average $\widehat c_{\mathfrak e}$
of independent $c_{\mathfrak e}$-coins and define the computable
proposal weight $a_{\mathfrak e}
 =k_{\mathfrak e}\exp\{L_0(\widehat c_{\mathfrak e}-1)\}.$
Keep these initial estimates fixed until one path is returned.
We propose a path with probability proportional to
$\prod_{\mathfrak e\in\gamma}a_{\mathfrak e}$.
To sample from this distribution without enumerating all complete
paths, let $D(v)$ be the sum of the proposal weights of all paths
from vertex $v$ to the terminal vertex $z$. Then these sums can be computed
backwards through the graph:
\[
 D(z)=1,\qquad
 D(v)=\sum_{\mathfrak e=(v,v')}a_{\mathfrak e}D(v').
\]
Here $D(z)=1$ is the weight of the empty path.
Starting from $o$, choose each outgoing edge
$\mathfrak e=(v,v')$ with probability
$a_{\mathfrak e}D(v')/D(v)$. Multiplication along the selected path
cancels the intermediate values of $D$, giving proposal probability
$\prod_{\mathfrak e\in\gamma}a_{\mathfrak e}/D(o)$.

To obtain the target distribution, the acceptance probability must
be proportional to the ratio of target and proposal weights:
\[
 \frac{\prod_{\mathfrak e\in\gamma}
                    k_{\mathfrak e}w_{\mathfrak e}}
      {\prod_{\mathfrak e\in\gamma}a_{\mathfrak e}}
 =
 \exp\left\{
 L_0\sum_{\mathfrak e\in\gamma}
              (c_{\mathfrak e}-\widehat c_{\mathfrak e})
 \right\}.
\]
The proportionality factor must be common to all candidate paths,
conditional on the initial estimates. We implement this exponential
acceptance rule using a Poisson number of independent factor coins.

We first describe one factor coin. Generate a fresh average
$\overline c_{\mathfrak e}$ of $B_0$ independent $c_{\mathfrak e}$-coins
at every edge and let
$d_{\mathfrak e}=\overline c_{\mathfrak e}-\widehat c_{\mathfrak e}$.
Conditional on the initial estimates, $d_{\mathfrak e}$ has mean
$c_{\mathfrak e}-\widehat c_{\mathfrak e}$. To turn these estimates
into a valid comparison probability, subtract the largest estimated
sum over all complete paths. This maximum is obtained by the
backward recursion
\[
 M(z)=0,\qquad
 M(v)=\max_{\mathfrak e=(v,v')}
                 \{d_{\mathfrak e}+M(v')\}.
\]
Thus $M(o)$ is at least the sum along the candidate path.
Using the same edge estimates for that sum and the maximum, define
\begin{equation}\label{br5:graph-factor}
 U_\gamma
 =1+\frac{\sum_{\mathfrak e\in\gamma}d_{\mathfrak e}-M(o)}{2N}.
\end{equation}
Since $d_{\mathfrak e}\in[-1,1]$ and every complete path has $N$
edges, $U_\gamma\in[0,1]$. A fresh uniform comparison with
$U_\gamma$ generates the factor coin.

For each proposed path, draw an independent
$K\sim\operatorname{Poisson}(\lambda_0)$, where
$\lambda_0=2NL_0$, and generate $K$ factor coins, using fresh edge
estimates and a fresh uniform variable for each factor.
Accept the path if every factor succeeds; an empty product succeeds.
Otherwise draw a new proposal path and repeat the Poisson test,
keeping the initial estimates fixed.

Conditional on the initial estimates, the expected maximum $M(o)$
is common to all candidate paths. The Poisson product therefore
gives an acceptance probability proportional to the required
exponential, with a common proportionality factor. The accepted
path consequently has weights exactly proportional to
$\prod_{\mathfrak e\in\gamma}k_{\mathfrak e}w_{\mathfrak e}$.
The sample sizes, coin implementations and proof are given in
Section~\ref{br5:weighted-proof}, which also establishes expected
work $O\{EN^3\log(4EN)\}$ when the input coins have uniformly bounded
conditional expected work.

\subsubsection{Acceptance within the selected cells}\label{br5:cell-acceptance}
Now we apply the weighted-path procedure to the cell graph and, conditional
on its selected path, draw the internal states independently and
uniformly in their cells. Write $\mathfrak e_i$ for the $i$th graph
edge of the selected path and $e_i=((i-1)h,z_{i-1};ih,z_i)$ for the
corresponding actual transition. As shown above, the resulting proposal
density on $E_0$ is $Z_0^{-1}\prod_{i=1}^N U_{\mathfrak e_i}$.
Accept this skeleton when independent coins of means
$p(e_i)/U_{\mathfrak e_i}$ succeed on all $N$ transitions.
Therefore, the product of the proposal density and the acceptance probability is
\[
 \frac{\prod_{i=1}^N U_{\mathfrak e_i}}{Z_0}
 \prod_{i=1}^N\frac{p(e_i)}{U_{\mathfrak e_i}}\ind_{E_0}
 =\frac{F\ind_{E_0}}{Z_0},
\]
as required by \eqref{br5:component-identities}.

For an actual transition $e$ represented by graph edge $\mathfrak e$,
it remains to implement its ratio coin without evaluating $p(e)$
or $U_{\mathfrak e}$. The cell comparisons give
\begin{equation}\label{br5:cell-slack}
 1+\frac{c_2}{N}\le\frac{U_{\mathfrak e}}{p(e)}\le1+\frac{C_2}{N},
 \qquad c_2=9/101,\quad C_2=1/9.
\end{equation}
The lower bound supplies a positive gap below one for each acceptance
probability; the upper bound ensures that all $N$ tests succeed with
probability at least $(1+C_2/N)^{-N}\ge e^{-C_2}$, uniformly over
the selected path and the states in its cells.

For an actual transition edge $e$, write
\[
 q=\frac{p(e)}{2g(e)},\qquad q_0=\frac{p(e^0)}{2g(e^0)},\qquad
 \beta=\frac{g(e)}{(1+1/(10N))g(e^0)}.
\]
Fresh sharp coins and linear Bernoulli factories produce coins of
means $A=\beta q$ and $q_0-A$. Choose one of these two labels with
equal probability and toss its coin, repeating the choice and toss
on failure. At the first success, return one exactly when the label
is $A$. The returned bit has mean $A/q_0=p(e)/U_{\mathfrak e}$.
Section~\ref{br5:cell-acceptance-proof} specifies the factory calls
and proves an $O(N)$ expected cost per ratio coin, hence $O(N^2)$
for the acceptance tests on one selected path. All these calls,
including the center coins, use fresh randomness after path selection.

\subsection{Components outside the graph}\label{br5:outside}

We next construct the trials for the events $B_i$ and $X_\zeta$
defined in Section~\ref{br5:component-plan}. Each trial starts with
a specified label, generates a candidate skeleton, and rejects it
unless its first violation matches that label. A further acceptance
test gives the corresponding densities in
\eqref{br5:component-identities}. 

Write $p_i=p((i-1)h,z_{i-1};ih,z_i)$ and $w_i=z_i-z_{i-1}$.
The known envelope masses are
\begin{equation}\label{br5:residual-masses}
 \begin{gathered}
 Q_h=K_{\mathrm b}(2\pi ch)^{-d/2},\qquad
 \delta_J=K_{\mathrm b}2^{d/2}e^{-9Dn/(64c)},\\
 V_-=-\Mc,\qquad V_+=\Mc+\Lambda/2,\qquad
 \delta_X=e^{-R_-+V_+T_0},\\
 D_i=\delta_JQ_h\quad(1\le i\le N),\qquad
 D_\zeta=\delta_XQ_h.
 \end{gathered}
\end{equation}
Here $Q_h$ bounds the final transition density, while $\delta_J$ and
$\delta_X$ account for the large increment or the exit, respectively.
The parameter choices make these contributions small enough relative
to $p_{T_0}$ to keep the total envelope mass controlled. Each known
probability below is implemented by an independent uniform comparison.

\subsubsection{A bad-edge component}\label{br5:bad-component}
For a specified label $B_i$ with $i<N$, use a broader Gaussian at
step $i$ to propose the large increment. Generate the internal states
$z_1,\ldots,z_{N-1}$ by successive exact short transitions, except
that step $i$ uses an independent Gaussian increment with density
$G_{2c,h}$. Attach $z_N=0$ and reject unless the skeleton belongs
to $B_i$.

For a retained skeleton, generate a transition coin on transition edge $i$ and
another on transition edge $N$. Draw independent uniforms $U_i,U_N$ for the
comparisons on these two edges. Conditional on the proposed states,
all four draws use fresh, mutually independent randomness. Accept
exactly when both coins return one and
\begin{equation}\label{br5:jump-tests}
 U_i\le\frac{K_{\mathrm b}G_{c,h}(w_i)}
                 {\delta_JG_{2c,h}(w_i)},\qquad
 U_N\le\frac{K_{\mathrm b}G_{c,h}(w_N)}{Q_h}.
\end{equation} If any of these tests fails, reject the proposal.

For $B_N$, generate the entire diffusion prefix by exact short
transitions, attach $z_N=0$, and retain only skeletons in $B_N$.
Accept when the last transition coin succeeds and an independent
uniform is at most $K_{\mathrm b}G_{c,h}(w_N)/(\delta_JQ_h)$.
The Gaussian bounds make every comparison probability
at most one. In either case the accepted density is
$F\ind_{B_i}/D_i$.

\subsubsection{An exit component}\label{br5:exit-component}
For a specified exit label $X_\zeta$, with $\zeta=(j,s)$, put
$\theta=se_j$. We propose the exit by changing the drift in this
direction. On $[0,T_0-h]$, use $b_\theta=b+a\theta$ until the
first coarse-grid observation outside $\mathcal B$,
\[
 \tau=\inf\{ih:1\le i\le N-1,\ X_{ih}\notin\mathcal B\},
 \qquad \inf\varnothing=\infty,
\]
and use the original drift thereafter. The diffusion coefficient
remains $\sigma$. The tilted coefficients have componentwise derivative
bounds through order $k$ at most $\Mc+d2^k\Mc^2$ and preserve the
ellipticity bounds. The short-transition sampler therefore applies
with constants for this fixed enlarged class. Its admissible duration
$h_{\mathrm{tilt}}$ is defined in Section~\ref{br5:parameter-choice},
where the common calibration ensures $h\le h_{\mathrm{tilt}}$.
All coefficient values are computed from original value queries.

The acceptance test for this change of drift requires additional path
observations. Generate an independent Poisson process of rate
$q_V=V_+-V_-$ on $[0,T_0-h]$. Merge its times with the coarse grid
and simulate the proposal at these times by exact short transitions.
The drift switches only at a coarse-grid exit, so an observation at a
Poisson time does not trigger a switch. Attach $z_N=0$ and retain
only skeletons in $X_\zeta$; in particular, reject if $\tau=\infty$.

Let $V(t,z)=\theta\cdot b(t,z)+\tfrac12\theta^\top a(t,z)\theta$.
For a retained proposal, draw independent uniforms for the comparisons
\begin{equation}\label{br5:exit-tests}
 \begin{aligned}
 U_1&\le e^{-\theta\cdot X_\tau+R_-},
 &U_2&\le e^{V_+(\tau-T_0)},\\
 U_t^{\mathrm P}&\le\frac{V(t,X_t)-V_-}{q_V}
 &&\text{at every Poisson time }t<\tau.
 \end{aligned}
\end{equation}
Conditional on the full proposal path and averaged over the Poisson
process and uniforms, the probability that these tests all succeed
is the likelihood of the original unconditional diffusion law relative to the
proposal law, divided by $\delta_X$.
Finally require the last transition coin and an independent uniform
comparison with $K_{\mathrm b}G_{c,h}(w_N)/Q_h$ to succeed.
All comparison probabilities lie in $[0,1]$ on the retained event.
After the auxiliary states are discarded, the accepted density is
$F\ind_{X_\zeta}/D_\zeta$. Section~\ref{br5:residual-proof} proves
both residual component identities, gives the finite classification
procedure, and bounds the work of their trials.

\subsubsection{Selecting a component and restarting}\label{br5:selection}
It remains to implement the selection probabilities of
Section~\ref{br5:component-plan}. The graph component has unknown
mass $Z_0$, so we cannot compute these probabilities directly.
Instead, append a path of $N$ edges for each label $B_i$ or $X_\zeta$
to the cell graph. Each added path has new internal vertices and shares
only the source and terminal vertices with the other paths. Give one
edge the known factor $D_i$ or $D_\zeta$, every other edge factor
one, and all its edges coins of mean one. Its path weight is exactly
the required component mass. 
Thus the total path weight of the augmented graph is
\[
 M_{\mathrm{env}}=Z_0+N\delta_JQ_h+2d\delta_XQ_h.
\]
Apply the weighted-path procedure of Section~\ref{br5:weighted}.
It selects an original cell path with its weight divided by
$M_{\mathrm{env}}$, or an added path with probability
$D_i/M_{\mathrm{env}}$ or $D_\zeta/M_{\mathrm{env}}$.
The original paths have total probability $Z_0/M_{\mathrm{env}}$;
conditional on this choice, their distribution is the one used in
the graph component. Thus a single graph draw selects both the
component and, when needed, its cell path.

For an original path, draw uniform states in its cells and perform
the acceptance tests of Section~\ref{br5:cell-acceptance}. For an
added path, run the proposal and tests of its specified component
once. If any required test fails, discard the entire candidate and
restart the augmented-graph draw with fresh initial estimates and
fresh component randomness. 

The component identities give accepted density
$F/M_{\mathrm{env}}$, as explained in
Section~\ref{br5:component-plan}. Section~\ref{br5:global-trial-proof}
proves $M_{\mathrm{env}}\le Cp_{T_0}$, uniformly in $n$ and the
coefficient instance. Hence each complete trial succeeds with
probability $p_{T_0}/M_{\mathrm{env}}\ge C^{-1}$, and repetition
until success returns the exact coarse bridge.

\subsection{Inserting the requested times}\label{br5:insert-section}
The auxiliary grid need not contain all requested observation times,
so we complete the bridge sample by inserting the missing states
from their conditional bridge laws.
Starting from the accepted coarse bridge, process the requested times
in increasing order, retaining all accepted
states and copying any state already present. To insert $t$, let
$(s,u)$ be the nearest generated state to its left (i.e. $X_s = u$) and $(r,v)$ the
coarse-grid state immediately to its right. Let
\[
 h'=r-s,\qquad \tau_1=t-s,\qquad \tau_2=r-t,\qquad w=v-u.
\]
Conditional on the retained states, the required density is
\[
 \frac{p(s,u;t,z)p(t,z;r,v)}{p(s,u;r,v)}.
\]
We sample this density by rejection, choosing the proposal according
to the displacement of the existing endpoints. Since $h'\le h$, both new edges
satisfy the short-time restrictions. For a candidate state $z$,
write $e_1=(s,u;t,z)$ and $e_2=(t,z;r,v)$.

\subsubsection{Moderate displacement}
Let $D_0<D_1$ and $C_g'>0$ be the fixed constants chosen in
Section~\ref{br5:parameter-choice}. If $|w|^2/h'\le D_0n$, put
$A_+=a(r,v)+\Id/n$ and partition candidate states into three disjoint sets which cover $\mathbb{R}^d$:
\[
 \begin{aligned}
 S_0&=\{|z-u|^2/\tau_1\le D_1n,\ |v-z|^2/\tau_2\le D_1n\},\\
 S_1&=\{|z-u|^2/\tau_1>D_1n\},\\
 S_2&=\{|z-u|^2/\tau_1\le D_1n,\ |v-z|^2/\tau_2>D_1n\}.
 \end{aligned}
\]
For a positive definite matrix $V$, $G_{V,h}$ denotes the centered
Gaussian density with covariance $hV$; it agrees with $G_{c,h}$ when
$V=c\Id$. On $S_0$, the covariance $A_+$ dominates both frozen
covariances. The other regions use a broader Gaussian on the edge
with large displacement. With
$\delta_1=K_{\mathrm b}2^{d/2}e^{-D_1n/(4c)}$, define
\begin{equation}\label{br5:insert-envelopes}
 \begin{aligned}
 \mathcal E_0(z)&=4(C_g')^2G_{A_+,\tau_1}(z-u)G_{A_+,\tau_2}(v-z),\\
 \mathcal E_1(z)&=\delta_1K_{\mathrm b}G_{2c,\tau_1}(z-u)G_{c,\tau_2}(v-z),\\
 \mathcal E_2(z)&=\delta_1K_{\mathrm b}G_{c,\tau_1}(z-u)G_{2c,\tau_2}(v-z).
 \end{aligned}
\end{equation}
On $S_j$, $\mathcal E_j$ envelopes the numerator of the conditional
bridge density. Its full-space mass is known:
\[
 M_0=4(C_g')^2G_{A_+,h'}(w),\quad
 M_1=\delta_1K_{\mathrm b}G_{c,2\tau_1+\tau_2}(w),\quad
 M_2=\delta_1K_{\mathrm b}G_{c,\tau_1+2\tau_2}(w).
\]
Choose $j$ with probability $M_j/(M_0+M_1+M_2)$ and draw $z$ from
the Gaussian density $\mathcal E_j/M_j$. Reject unless $z\in S_j$.
The Gaussian means and covariances are recorded in
Section~\ref{br5:insertion-proof}.

For a retained state with $j=0$, require both sharp coins and
independent comparisons with
\[
 \frac{g(e_1)}{C_g'G_{A_+,\tau_1}(z-u)},\qquad
 \frac{g(e_2)}{C_g'G_{A_+,\tau_2}(v-z)}
\]
to succeed. If $j=1$ or $2$, require both transition coins and one
comparison with, respectively,
\[
 \exp\left\{\frac{D_1n-|z-u|^2/\tau_1}{4c}\right\},\qquad
 \exp\left\{\frac{D_1n-|v-z|^2/\tau_2}{4c}\right\}.
\]
These tests have product mean $p(e_1)p(e_2)/\mathcal E_j(z)$ on
$S_j$. Hence the accepted density is
$p(e_1)p(e_2)/(M_0+M_1+M_2)$.

\subsubsection{Larger displacement and insertion cost}
If $|w|^2/h'>D_0n$, draw
\[
 z\sim\calN\left(\frac{\tau_2u+\tau_1v}{h'},
                    \frac{c\tau_1\tau_2}{h'}\Id\right)
\]
and accept when both $\TC(s,u;t,z)$ and $\TC(t,z;r,v)$ succeed.
The accepted density is
$p(e_1)p(e_2)/(K_{\mathrm b}^2G_{c,h'}(w))$, so this also gives
the required conditional bridge law.

In either case, rejection repeats only the current insertion trial
with fresh randomness, including a new mixture choice when applicable.
The existing states remain fixed. Moderate endpoint displacements give
uniformly bounded conditional expected cost. For larger displacements,
the conditional bound is $CG_{c,h'}(w)/p(s,u;r,v)$; averaging under
the exact bridge law cancels its denominator. The analysis in
Section~\ref{br5:insertion-proof} therefore gives total additional
expected work $O(m+N)$.

\subsection{The bridge sampler and its work}
\label{bridge:general-construction}
The preceding construction handles bridges with zero endpoints on
$[0,2n]$; we now extend it to arbitrary horizons and endpoints.
For general $T,x,y$, subtract the line joining the endpoints and
rescale time and space. Write $w=y-x$ and set
\begin{equation}\label{main:exponents}
 H_{T,x,y}=1+T+|w|^2/T,\qquad Q=16+d(4+8/\alpha).
\end{equation}
Choose the computable class constant $L_{\mathrm{sc}}\ge1$ specified
in Section~\ref{br5:bridge-sampler-proof} and set
\begin{equation}\label{br:class-preserving-scale}
 n=\lceil L_{\mathrm{sc}}H_{T,x,y}\rceil,\qquad S=2n,\qquad r=T/S.
\end{equation}
The transformation
\begin{equation}\label{br:affine-process}
 Z_s=\frac{X_{rs}-x-(s/S)w}{\sqrt r},\qquad 0\le s\le S,
\end{equation}
gives zero endpoints and coefficients
\begin{equation}\label{br:affine-coefficients}
 \begin{aligned}
 \widetilde b(s,z)
 &=\sqrt r\,b(rs,x+(s/S)w+\sqrt r\,z)-\frac{w}{\sqrt{TS}},\\
 \widetilde\sigma(s,z)
 &=\sigma(rs,x+(s/S)w+\sqrt r\,z).
 \end{aligned}
\end{equation}
They satisfy the original class bounds and can be evaluated using
original coefficient queries. The following algorithm collects the
construction and its restart rules. The theoretical analysis can be found in Section~\ref{br5:bridge-sampler-proof}.

\begin{algorithm}
\algcaption{Bridge sampler}\label{bridge:general-algorithm}
\alginout{$T>0$, endpoints $x,y$, and $0<t_1<\cdots<t_m<T$.}
{the bridge target \eqref{eq:bridgetarget}.}
\begin{enumerate}
\item Compute \eqref{br:class-preserving-scale} and provide the
transformed queries \eqref{br:affine-coefficients}. Transform the
requests to $t_i/r$ and include $n$ if absent. Use these coefficients
and $T_0=2n$ in steps~2--6.
\item Calibrate the class constants as in
Section~\ref{br5:parameter-choice}. Form the coarse grid and cell
graph of Section~\ref{br5:proposal}, then add the component paths
as in Section~\ref{br5:selection}.
\item Draw one augmented-graph path by the weighted-path procedure.
Generate fresh initial estimates for this call, keeping them fixed
only during its internal retries.
\item For an original path, draw uniform states in its cells and
perform the within-cell acceptance tests. For an added path, run one trial
of its bad-edge or exit component.
\item If a required test fails, discard the candidate and return to
step~3 with fresh randomness. Otherwise retain the coarse bridge.
\item Insert the requested times by
Section~\ref{br5:insert-section}, retaining accepted states and
repeating only the current insertion trial upon rejection.
\item At each original requested time return
\begin{equation}\label{br:affine-output}
 x+(t_i/T)w+\sqrt r\,Z_{t_i/r},\qquad 1\le i\le m,
\end{equation}
and discard all other states.
\end{enumerate}
\end{algorithm}

\begin{theorem}[Bridge sampling]\label{bridge:general}
Algorithm~\ref{bridge:general-algorithm}, denoted by $A_B$, is
query-only, terminates almost surely and samples
\eqref{eq:bridgetarget} exactly for every $T>0$, $x,y\in\R^d$,
$m\ge1$ and requested grid $0<t_1<\cdots<t_m<T$. There is a constant
$C>0$, depending only on the fixed coefficient class and primitive
costs, such that
\begin{equation}\label{br:general-work}
 \sup_{(b,\sigma)\in\mathcal F_T}\E\Wk_{A_B}
 \le C\{H_{T,x,y}^Q\log(2+H_{T,x,y})+m\}.
\end{equation}
The constant is uniform in the horizon, endpoints and requested grid.
\end{theorem}

\section{Minimax optimality}\label{sec:minimax}
The preceding sections construct exact unconditional diffusion and bridge samplers.
We now show that these constructions are minimax optimal over all exact query-only
algorithms under the coefficient class and work model of
Section~\ref{sec:setup}: jointly
in the horizon $T$ and the number of requested observations $m$
for unconditional sampling, and in $m$ for bridge sampling with fixed
horizon and endpoints. 

We begin with the minimax formulation. Fix $d,k,\Mc,\lambda,\Lambda$ and the primitive costs, and assume
that the coefficient class $\mathcal F_T$ of
Section~\ref{sec:setup} is nonempty. For $T>0$ and an integer $m\ge1$,
we define the class of requested observations
\begin{align}
 \mathcal G_F(T,m)
 &=\{(t_1,\ldots,t_m):0<t_1<\cdots<t_m\le T\},\nonumber\\
 \mathcal G_B(T,m)
 &=\{(t_1,\ldots,t_m):0<t_1<\cdots<t_m<T\}.
 \label{cx:grids}
\end{align}

\begin{definition}\label{cx:admissible}
An admissible algorithm receives the class bounds, horizon, endpoint
data and requested grid, uses only the coefficient queries and
primitives specified in Section~\ref{sec:model}, and terminates almost
surely with an exact sample for every allowed choice of coefficients
and inputs in its target family. Its expected work may be infinite.
For fixed endpoints, define
\begin{align}
 \mathfrak R_F^x(T,m)
 &=\inf_A\sup_{\substack{(b,\sigma)\in\mathcal F_T\\
                         \mathbf t\in\mathcal G_F(T,m)}}
       \E_{b,\sigma}\Wk_A(T,x,\mathbf t),\nonumber\\
 \mathfrak R_B^{x,y}(T,m)
 &=\inf_A\sup_{\substack{(b,\sigma)\in\mathcal F_T\\
                         \mathbf t\in\mathcal G_B(T,m)}}
       \E_{b,\sigma}\Wk_A(T,x,y,\mathbf t).
 \label{cx:risks}
\end{align}
The infima are over admissible query-only algorithms for the respective
families. A single algorithm within each supremum must handle all
coefficients and grids in that supremum; it accesses the unknown
coefficients only through point queries.
\end{definition}

The following theorem (proof can be found in Section \ref{app:minimax-proofs}) compares the work of our algorithms with these
minimax quantities. The unconditional comparison allows both $T$ and $m$
to vary. For bridges, we fix $T,x,y$ and compare the dependence on $m$.

\begin{theorem}[Minimax guarantees]\label{main:rates}\label{cx:rates}
The algorithms $A_F$ and $A_B$ of Theorems~\ref{fw:upper}
and~\ref{bridge:general} are admissible. There are constants $c,C>0$,
depending only on the fixed coefficient class and primitive costs,
such that the following hold.
\begin{enumerate}
\item[\rm(i)] For every $x\in\R^d$, $T>0$ and $m\ge1$,
\begin{equation*}
 \begin{aligned}
 c(T+m)&\le\mathfrak R_F^x(T,m)\le\sup_{\substack{(b,\sigma)\in\mathcal F_T\\
                      \mathbf t\in\mathcal G_F(T,m)}}
       \E_{b,\sigma}\Wk_{A_F}(T,x,\mathbf t)
 \le C(T+m).
 \end{aligned}
\end{equation*}
Thus $A_F$ is minimax optimal jointly in $T$ and $m$.

\item[\rm(ii)] For every fixed $T>0$ and $x,y\in\R^d$, there is
a finite constant $C_{T,|y-x|}>0$, depending only on $T$, $|y-x|$,
the coefficient class and primitive costs, such that for all $m\ge1$,
\begin{equation}\label{cx:bridge-rate}
 \begin{aligned}
 cm&\le\mathfrak R_B^{x,y}(T,m)\le\sup_{\substack{(b,\sigma)\in\mathcal F_T\\\mathbf t\in\mathcal G_B(T,m)}}
       \E_{b,\sigma}\Wk_{A_B}(T,x,y,\mathbf t)
 \le C_{T,|y-x|}m.
 \end{aligned}
\end{equation}
Consequently, $A_B$ is minimax optimal in $m$ for fixed $T,x,y$.
\end{enumerate}
\end{theorem}

\begin{remark}
For bridges, the minimax optimality established here concerns the
number of requested observations $m$, with the horizon and endpoints
fixed. The construction also improves the dependence on the horizon
relative to the natural terminal-rejection method of
Section~\ref{br5:terminal-rejection}. For equal endpoints, that
method has worst-case expected work at least $c_1e^{c_2T}$, as
witnessed by an admissible constant-coefficient diffusion, whereas
our sampler has the uniform bound $O\{(1+T)^Q\log(3+T)+m\}.$
Thus the horizon dependence is polynomial up to a logarithmic
factor. The sharp minimax dependence on the bridge horizon remains
open and is a direction of future research.
\end{remark}

\section{Conclusions and discussions}\label{sec:discussion}
We formulate a query-only computational model for exact diffusion
simulation and construct unconditional diffusion and bridge samplers
with finite expected total work. The algorithms access the
coefficients through point evaluations, and the work model accounts
for all operations, including preprocessing, coefficient queries,
auxiliary randomization, rejected proposals and output. The
constructions apply to uniformly elliptic multivariate diffusions
with bounded, time- and state-dependent coefficients under the
stated regularity assumptions, allowing nonconstant covariance and
nongradient drift. The unconditional sampler concatenates exact
short transitions, whose acceptance decisions are implemented by
transition coins obtained from randomized parametrix expansions.
For bridges, we develop a proposal based on a finite weighted graph
of spatial cells and time layers that incorporates both prescribed
endpoints. Coins for the unknown edge weights permit exact
graph-path sampling, followed by uniform sampling within the
selected cells and acceptance--rejection sampling. Additional
proposal components cover large increments and spatial exits,
preserving the full support of the continuous-state bridge.
For fixed endpoints and a fixed number of requested observations,
the resulting expected work grows at most polynomially as the
horizon increases, improving on the natural terminal-rejection
construction, whose expected work can grow exponentially.
Subsequent insertion of the requested observations adds linear
expected work without a minimum spacing condition. Matching lower
bounds in the same coefficient class and computational model
establish minimax optimality jointly in the horizon $T$ and the
number of requested observations $m$ for unconditional sampling,
and in $m$ for bridge sampling with fixed horizon and endpoints.

The sharp joint dependence of bridge sampling work on the horizon
and the number of requested observations remains open, as does
the optimal dependence on the endpoints. Reducing the cost of
graph construction and path sampling, together with stronger
information lower bounds, may help resolve these questions.
Further directions include applications of the exact samplers to
inference for discretely observed diffusions, conditional
simulation and stochastic optimization.

\appendix

\section{Work accounting}
\label{app:accounting}\label{supp:sup:section4}

The computational conventions in Section~\ref{sec:model} charge a fixed
positive finite cost for each scalar primitive. Matrix operations are
expanded into scalar operations, all of which are charged. Throughout the analysis,
the dimension $d$ is treated as a fixed parameter. Integer budgets
and ceilings are computed by counted loops. A product with a random
number of factors charges each multiplication and the work of generating
each factor. We generate Poisson and geometric random variables using the basic
random draws and operations listed above, which are all included in the total work. Known class constants are calibrated once and their finite preprocessing cost is included.

Allocation and disposal of temporary states, index updates, rejected
trials and every recursive call count toward $\Wk$. Distinct calls use
fresh randomness unless shared samples are explicitly required. A coin
with a known bias is obtained from an independent uniform comparison. The following lemma accounts for all trials of a rejection procedure,
including the successful trial.

\begin{lemma}[Work through the first successful trial]\label{supp:fw:retry}
Conditional on a fixed history $\mathcal H$, let
$(C_k,S_k,Y_k)_{k\ge1}$ be independent identically distributed trial records,
where $C_k\ge0$ is the work of trial $k$, $S_k\in\{0,1\}$ indicates success (or not), and $Y_k$ is the proposed output. Suppose
$\Pp(S_k=1\mid\mathcal H)=q>0$ and
$\E(C_k\mid\mathcal H)\le c$. Then, for
$\nu=\inf\{k:S_k=1\}$, we have $\nu<\infty$ almost surely and $\E\left[\sum_{k=1}^{\nu} C_k\,\middle|\,\mathcal H\right]\le c/q.$
The successful trial's output has the one-trial conditional law given
success.
\end{lemma}
\begin{proof}
Conditional on $\mathcal H$, $\nu$ is geometric with parameter $q$.
The event $\{\nu\ge k\}$ depends only on the first $k-1$ trials and is
conditionally independent of $C_k$. Then by the conditional Tonelli theorem, we have
\[
 \E\left[\sum_{k\ge1}\ind\{\nu\ge k\}C_k
       \,\middle|\,\mathcal H\right]
 =\sum_{k\ge1}(1-q)^{k-1}\E(C_k\mid\mathcal H)\le c/q.
\]
If $Y_k$ is the output in trial $k$, then for a measurable set $A$,
\[
 \Pp(Y_\nu\in A\mid\mathcal H)
 =\sum_{k\ge1}(1-q)^{k-1}
       \Pp(Y_k\in A,S_k=1\mid\mathcal H)
 =\Pp(Y_1\in A\mid S_1=1,\mathcal H),
\]which finishes the proof.
\end{proof}

\section{Details of exact diffusion sampling}
\label{app:forward-proofs}\label{supp:ana:foundation}

We first deduce the unconditional sampling guarantee from the short-transition routine.
The subsequent lemmas establish the representation, randomization and
coin constructions on which that routine depends. All constants are
computed from the stated class bounds, and all work is measured using
Appendix~\ref{app:accounting}.

\begin{proof}[Proof of Theorem~\ref{fw:upper}]
Lemma~\ref{core:shorttransition} gives the exact conditional law and a
uniform conditional expected-work bound for every short step in
Algorithm~\ref{fw:algorithm}. Successive conditioning and
Chapman--Kolmogorov identity, integrating out auxiliary states, therefore give
\eqref{eq:forwardtarget}. The routines use only coefficient values.
Their constants are calibrated on unit windows from the class bounds;
this has fixed cost and requires no coefficient queries outside $[0,T]$,
including when $T<1$.

The integers $n_i$ can be computed by counting loops in
$O(N_{\mathrm{short}})$ work. Forming the intermediate times,
managing states and writing outputs costs
$O(N_{\mathrm{short}}+m)$. Summing the conditional work bounds and
using \eqref{fw:counts} gives
\[
 \E\Wk_{A_F}\le C_0(1+N_{\mathrm{short}}+m)
 \le C(t_m+m),
\]
since $h_{\mathrm{env}}$ is a class constant and $m\ge1$.
\end{proof}

\subsection{The parametrix representation}\label{sec:B.1}

\begin{lemma}[Parametrix representation]\label{core:parametrix}
\label{supp:core:parametrix}
Under Assumption~\ref{ass:class}, for every fixed
$T_{\mathrm{win}}>0$ there are computable constants $C_g,C_H,C_p^*>0$
such that, for $0<h=t-s\le\min\{T_{\mathrm{win}},T\}$,
\eqref{core:parametrix-series} converges absolutely to $p(e)$ and
\[
 g(e)\le C_gG_{\kappa,h}(w),\qquad
 |H(e)|\le C_Hh^{-1/2}G_{\kappa,h}(w),\qquad
 0\le R(e)\le C_p^*,\qquad \kappa=2\Lambda.
\]
The constants depend only on $d,\Mc,\lambda,\Lambda,T_{\mathrm{win}}$.
As a result, the
Chapman--Kolmogorov identity holds pointwise.
\end{lemma}

For the proof, write $\|\cdot\|_F$ for the Frobenius norm and
$P_{s,u;t}$ for the transition probability from $(s,u)$ to time $t$.
Set $K_0=g$ and $K_j=g\otimes H^{\otimes j}$ for $j\ge1$.

\begin{proof}
Fix $T_{\mathrm{win}}>0$ and restrict attention to $h\le\min\{T_{\mathrm{win}},T\}$.
Throughout the proof, we suppress the dependence of $C_H$ and $C_p^*$
on $T_{\mathrm{win}}$ in the notation.

The product rule gives $|\partial_\xi a_{ij}|\le2d\Mc^2$ for every
time or spatial coordinate $\xi$. Hence
\[
 |b(s,u)|\le B_b:=\sqrt d\,\Mc,\qquad
 \|a(s,u)-a(t,v)\|_F\le L_a(h+|w|),
\] where $L_a:=2d^{5/2}\Mc^2.$
With $r=|w|/\sqrt h$ and $\kappa=2\Lambda$, ellipticity yields $\frac{g(e)}{G_{\kappa,h}(w)}\le C_g e^{-r^2/(4\Lambda)} \leq C_g,$ with $C_g:=\left(\frac{2\Lambda}{\lambda}\right)^{d/2}.$
Using $|A^{-1}w/h|\le r/(\lambda\sqrt h)$ in
\eqref{core:defect}, we obtain $|H(e)|\le h^{-1/2}P_{T_{\mathrm{win}}}(r)g(e),$ where
\[
 P_{T_{\mathrm{win}}}(r):=\frac{B_b}{\lambda}r+
 \frac{L_a}{2}(r+\sqrt{T_{\mathrm{win}}})
 \left(\frac{r^2}{\lambda^2}+\frac{\sqrt d}{\lambda}\right)
 =\sum_{j=0}^3\alpha_jr^j.
\]
Since $\sup_{r\ge0}r^je^{-r^2/(4\Lambda)}=(2j\Lambda/e)^{j/2}$
for $j\ge1$, the computable choice
\[
C_H:=C_g\left[\alpha_0+
\sum_{j=1}^3\alpha_j(2j\Lambda/e)^{j/2}\right]
\]
gives $g(e)\le C_gG_{\kappa,h}(w)$ and $|H(e)|\le C_Hh^{-1/2}G_{\kappa,h}(w).$

For a bounded Lipschitz function $f$ and time pair $q < t$, we denote $F_{t,f}(q,z)=\int g(q,z;t,v)f(v)\,\dd v.$
Then we have
\begin{equation*}
 \sup_z|F_{t,f}(t-h,z)-f(z)|
 \le\sqrt h\left(C\|f\|_\infty+
 \sqrt{d\Lambda}\operatorname{Lip}(f)\right),
\end{equation*}
where $C$ depends only on $d,\lambda,\Lambda,L_a$.

For $f\in C_c^\infty(\R^d)$ and $s<\beta<t$, the initial-variable
derivatives of $F_{t,f}$ are bounded on $[s,\beta]\times\R^d$.
It\^o's formula and \eqref{core:defect} therefore imply
\[
 \E F_{t,f}(\beta,X_\beta)-F_{t,f}(s,u)
 =\int_s^\beta\int P_{s,u;q}(\dd z)
       \int H(q,z;t,v)f(v)\,\dd v\,\dd q.
\]
The terminal approximation and continuity of $X$ identify the limit
of the left side as $\beta\uparrow t$. The absolute integral on the
right is at most $2C_H\sqrt h\,\|f\|_\infty$. Therefore, by Fubini's theorem and dominated convergence theorem, 
\begin{equation}\label{supp:ana:measure-duhamel}
 P_{s,u;t}(\dd v)=\left[g(s,u;t,v)+
 \int_s^t\int H(q,z;t,v)P_{s,u;q}(\dd z)\,\dd q\right]\dd v.
\end{equation}

For ordered times $s<q_1<\cdots<q_n<t$, set $q_0=s$, $q_{n+1}=t$
and $\Delta_i=q_{i+1}-q_i$.
Gaussian convolution in the definition of $K_n$, followed by the
Dirichlet integral, gives
\begin{align}
 \mathsf Z_n(h)
&=\int_{\substack{\Delta_1,\ldots,\Delta_n>0\\\Delta_1+\cdots+\Delta_n<h}}\prod_{i=1}^n\Delta_i^{-1/2}\,\dd\Delta_1\cdots\dd\Delta_n
   =\frac{\pi^{n/2}h^{n/2}}{\Gamma(1+n/2)}
 \label{supp:ana:term-bound}
\end{align} and $|K_n(e)|\le C_gC_H^n\mathsf Z_n(h)G_{\kappa,h}(w).$
In particular, all convolution integrals are absolutely convergent.
Iterating \eqref{supp:ana:measure-duhamel} gives
\[
 P_{s,u;t}(\dd v)=
 \left[\sum_{n=0}^{m-1}K_n(s,u;t,v)+\mathcal R_m(s,u;t,v)\right]\dd v,
 \qquad
 \|\mathcal R_m(s,u;t,\cdot)\|_1\le C_H^m\mathsf Z_m(h).
\]
Since $\frac{C_H^{n+2}\mathsf Z_{n+2}(h)}{C_H^n\mathsf Z_n(h)}
 =\frac{\pi C_H^2h}{1+n/2}\longrightarrow0,$
the remainder tends to zero and the series converges absolutely
and in $L^1(\dd v)$.
Also, each $K_n$ is jointly continuous for $s<t$.
Then by the bound \eqref{supp:ana:term-bound} the parametrix series is locally uniformly convergent on $s<t$.
Its sum is therefore continuous and, by the vanishing remainder,
agrees with the continuous transition density at every endpoint: $p(e)=g(e)+\sum_{n\ge1}K_n(e).$
Hence,
$p=g+p\otimes H$ pointwise.

Finally, let $u_h=C_H\sqrt{\pi h}$. Splitting the series into even
and odd terms and using
$\Gamma(j+3/2)\ge(\sqrt\pi/2)j!$ gives $\sum_{n\ge0}\frac{u_h^n}{\Gamma(1+n/2)}
 \le(1+2u_h/\sqrt\pi)e^{u_h^2}.$
Consequently, with $u_*:=C_H\sqrt{\pi T_{\mathrm{win}}},$
$0\le R(e)\le C_p^*:=C_g(1+2u_*/\sqrt\pi)e^{u_*^2}.$
All constants depend only on $d,\Mc,\lambda,\Lambda,T_{\mathrm{win}}$.
Taking $T_{\mathrm{win}}=1$ gives the constants used in the main text.
The Markov property implies the Chapman--Kolmogorov identity almost
everywhere in the final endpoint. Continuity and the Gaussian bound
make its integral continuous there, so the identity holds pointwise.
\end{proof}

\subsection{Random corrections and the local estimator}\label{sec:B.2}

\begin{lemma}[Randomized convolution term]\label{core:randomterm}
For every $j\ge1$ and $0<h\le1$, a query-only construction generates
$\Upsilon_j$ satisfying \eqref{core:Tnproperties} in $O(j+1)$ work,
uniformly in the edge and coefficient instance.
\end{lemma}
\begin{proof}
We first describe the algorithm, and then show that is satisfies \eqref{core:Tnproperties} with the work guarantee. Draw independent $U_0\sim\operatorname{Uniform}(0,1)$ and
$\chi_i\sim\calN(0,1)$, $1\le i\le j$, and set
\begin{equation}\label{core:randomgaps}
 Q_0=-\log U_0,\qquad Q_i=\chi_i^2/2\ (1\le i\le j),
 \qquad \Delta_i=\frac{hQ_i}{\sum_{l=0}^jQ_l}.
\end{equation}
On the null event that any gap is zero, return zero before forming
denominators. Otherwise the normalized gaps have the Dirichlet law
with parameters $(1,1/2,\ldots,1/2)$; their density on the gap simplex
is $\mathsf Z_j(h)^{-1}\prod_{i=1}^j\Delta_i^{-1/2}$.
Let $\tau_0=s$, $\tau_{i+1}=\tau_i+\Delta_i$, $z_0=u$ and
$z_{j+1}=v$. Conditional on the gaps, generate the intermediate
states of a Gaussian bridge with covariance parameter $\kappa\Id$.
Its joint density is
\[
 \frac{\prod_{i=0}^jG_{\kappa,\Delta_i}(z_{i+1}-z_i)}
 {G_{\kappa,h}(v-u)}.
\]
For an explicit implementation, draw $z_i$, $1\le i\le j$,
successively with conditional mean and covariance
\[
 z_{i-1}+\frac{\tau_i-\tau_{i-1}}{t-\tau_{i-1}}(v-z_{i-1}),
 \qquad
 \frac{(\tau_i-\tau_{i-1})(t-\tau_i)}{t-\tau_{i-1}}\kappa\Id.
\]
Writing $G_i=G_{\kappa,\Delta_i}(z_{i+1}-z_i)$, define
\begin{equation}\label{core:random-correction}
 \Upsilon_j=\mathsf Z_j(h)
 \frac{g(s,u;\tau_1,z_1)}{G_0}
 \prod_{i=1}^j
 \frac{\sqrt{\Delta_i}H(\tau_i,z_i;\tau_{i+1},z_{i+1})}{G_i}.
\end{equation}
Cancellation against the gap and bridge densities proves the
expectation in \eqref{core:Tnproperties}. The bounds of
Lemma~\ref{core:parametrix} applied to each factor give its pathwise
bound. Absolute integrability follows from the same bounds.

All gaps, bridge states and factors require $O(j+1)$ scalar work in
fixed dimension. The normalizers are evaluated without the Gamma
function using the quantity
\[
 \mathsf Z_0(h)=1,\qquad \mathsf Z_1(h)=2\sqrt h,
 \qquad \mathsf Z_{j+2}(h)=\frac{2\pi h}{j+2}\mathsf Z_j(h).
\]Therefore, the proof is finished.
\end{proof}

\begin{lemma}[Local estimator]\label{core:localpositive}
The local estimator $\hat{R}_{\mathrm{loc}}$ is unbiased for $R(e)$
and has uniformly bounded expected work for $0<h\le1$.
For each $L\ge1$, there is a computable $h(L)>0$ such that
$h\le h(L)$ and $|w|\le L\sqrt h$ imply \eqref{core:localbounds}.
The work at selected order $J$ is $O(J+1)$.
\end{lemma}
\begin{proof}
Generate $J$ by counted fair-coin trials until the first success;
then $\Pp(J=j)=2^{-j}$ and $\E J=2$.
The absolute convolution bounds and
Lemma~\ref{core:randomterm} justify summing the conditional
expectations, giving $\E\widehat R_{\mathrm{loc}}=R(e)$.
They also give the asserted work bound.

For the pathwise assertion, ellipticity gives
\[
 r_0\ge\rho_L:=(\kappa/\Lambda)^{d/2}e^{-L^2/(2\lambda)}
 \quad\text{if }|w|\le L\sqrt h.
\]
Choose
\begin{equation*}
 \delta_L=\min\{1/2,\rho_L/(8C_g)\},\qquad
 h(L)=\min\{1,\delta_L^2/[4\pi(1+C_H)^2]\}.
\end{equation*}
Then $2C_H\sqrt{\pi h}\le\delta_L$. Since $J\ge1$ and
$\Gamma(1+J/2)\ge1/2$, \eqref{core:Tnproperties} yields
\[
 |2^J\Upsilon_J|
 \le\frac{C_g(2C_H\sqrt{\pi h})^J}{\Gamma(1+J/2)}
 \le2C_g\delta_L\le\rho_L/4\le r_0/4.
\]
Together with $r_0\le C_g$, this proves \eqref{core:localbounds}.
\end{proof}

\subsection{Calibration and the spatial coin}\label{supp:spatialproof}

We now specify the constants in \eqref{core:spatial-envelope}. Put
\[
 C_{1e}=(8/7)^{d/2},\qquad C_{2e}=(4/3)^{d/2},
 \qquad\beta_0=4/7,\qquad\eta=1/4,
\]
and choose
\begin{equation}\label{core:envelope-calibration}
 \begin{gathered}
 b_{\mathrm{env}}=\max\{2,B_{\mathrm{loc}},C_p^*\},\qquad
 k_{\mathrm{env}}=\eta/(3C_{2e}),\\
 L_{\mathrm{env}}^2=\max\left\{1,
 \frac1{a_0}\log\frac{3b_{\mathrm{env}}^2}{\eta k_{\mathrm{env}}},
 \frac1{(1-\beta_0)a_0}\log\frac{6b_{\mathrm{env}}C_{1e}}\eta
 \right\},\qquad h_{\mathrm{env}}=h(L_{\mathrm{env}}).
 \end{gathered}
\end{equation}
Here the parametrix constants are calibrated with $T_{\mathrm{win}}=1$.
Then
\begin{equation}\label{core:midpoint-normalizer}
 S(\zeta)=b_{\mathrm{env}}^2
 +2b_{\mathrm{env}}k_{\mathrm{env}}C_{1e}
          e^{\beta_0a_0|\zeta|^2}
 +k_{\mathrm{env}}^2C_{2e}e^{a_0|\zeta|^2}.
\end{equation}
The tilted density in Algorithm~\ref{core:spatialalgorithm} is sampled
by choosing a Gaussian component with probability proportional to its
weight in the following array:
\[
\begin{array}{c|c|c}
 \text{Weight}&\text{Mean}&\text{Covariance}\\ \hline
 b_{\mathrm{env}}^2&0&\Id\\
 b_{\mathrm{env}}k_{\mathrm{env}}C_{1e}e^{\beta_0a_0|\zeta|^2}
 &\zeta/(7\sqrt\kappa)&(8/7)\Id\\
 b_{\mathrm{env}}k_{\mathrm{env}}C_{1e}e^{\beta_0a_0|\zeta|^2}
 &-\zeta/(7\sqrt\kappa)&(8/7)\Id\\
 k_{\mathrm{env}}^2C_{2e}e^{a_0|\zeta|^2}&0&(4/3)\Id.
\end{array}
\]
The weights sum to $S(\zeta)$. Thus this draw has bounded work.

\begin{lemma}[Spatial coin]\label{core:spatialguarantee}
With the choices in \eqref{core:envelope-calibration},
Algorithm~\ref{core:spatialalgorithm} is query-only, terminates almost
surely and satisfies
\[
 \Pp\{\operatorname{SpatialCoin}(e)=1\}=R(e)/C(\zeta),
 \qquad \E W_{\mathrm{sp}}(e)\le C_{\mathrm{sp}}<\infty,
\]where $C_{\mathrm{sp}}$ is a constant that depends only on $d,\Mc,\lambda,\Lambda$ and the primitive costs.
The bound is uniform over the coefficient class, time locations and
endpoints with $0<t-s\le h_{\mathrm{env}}$. It also holds
conditionally when the input edge is determined by previous history
and the call uses fresh randomness.
\end{lemma}
\begin{proof}
The choices of $b_{\mathrm{env}}$ and $h_{\mathrm{env}}$ imply
$0\le R(e)\le C(\zeta)$ and make the local uniform comparison valid.
Dividing the three terms of \eqref{core:midpoint-normalizer} by
$C(\zeta)\ge k_{\mathrm{env}}e^{a_0|\zeta|^2}$ bounds each by
$\eta/3$ whenever $|\zeta|>L_{\mathrm{env}}$. Thus every nonlocal
call branches with probability at most $1/4$.

Let $N_l$ be the number of calls at depth $l$, with $N_0=1$.
Conditional expectation gives
\[
 \E N_{l+1}\le\tfrac12\E N_l,\qquad
 \E N_l\le2^{-l},\qquad
 \E\sum_{l\ge0}N_l\le2.
\]
Consequently the recursion has finitely many calls almost surely.
Each call's own operations terminate almost surely and have uniformly
bounded conditional mean work: the local case costs $O(J+1)$ with
$\E J=2$, and the nonlocal case uses a bounded number of operations,
including the mixture draw. Charging a call only for its own work
and summing over depths therefore proves the uniform work bound.

For correctness, let $f(e)=R(e)/C(\zeta)$. The local output has mean
$f(e)$ by Lemma~\ref{core:localpositive}. On a nonlocal edge, write
$\widetilde\E$ for expectation under the tilted midpoint law. Therefore,
\[
 \frac{S(\zeta)}{C(\zeta)}\widetilde\E[f(e_1)f(e_2)]
 =\frac{\E[R(e_1)R(e_2)]}{C(\zeta)}=f(e).
\]
To avoid assuming correctness of the recursive outputs, truncate at
depth $l$ and replace unresolved calls by auxiliary coins of their
prescribed means $f$, independently conditional on their inputs.
Such coins are well defined because $0\le f\le1$. Backward induction
gives mean $f(e)$ for each truncated procedure. Couple it with the
full procedure using the same randomness before depth $l$. Their
outputs differ only if that depth is reached, with probability at
most $\Pp(N_l>0)\le2^{-l}$. Letting $l\to\infty$ proves the claim.
Fresh randomness gives the same argument conditionally on any
history determining the input edge.
\end{proof}

\subsection{Transition coins and short transitions}

\begin{lemma}[Transition coin]\label{core:shortcoin}
The transition coin algorithm is query-only, terminates
almost surely and satisfies \eqref{core:shortcoinlaw}. Its total
expected work is bounded by a class constant $C_{\mathrm{TC}}$,
uniformly over all edges of duration at most $h_{\mathrm{env}}$.
The same bound holds conditionally on the input edge.
\end{lemma}
\begin{proof}
Since $C(\zeta)\le A_0e^{a_0|\zeta|^2}$, the additional comparison
is valid. Independence and Lemma~\ref{core:spatialguarantee} imply
\[
 \Pp\{\TC(e)=1\}
 =\frac{R(e)}{C(\zeta)}
   \frac{C(\zeta)}{A_0e^{a_0|\zeta|^2}}
 =\frac{p(e)}{E_h(w)}.
\]
The Gaussian identity \eqref{core:gaussianenvelope} follows from
$a_0=1/(8\kappa)$. The additional operations have bounded cost,
so the work and termination assertions follow from the spatial coin.
\end{proof}

\begin{lemma}[Short transition]\label{core:shorttransition}
The short transition algorithm in
Section~\ref{core:shorttransitionalgorithm} is query-only, terminates
almost surely and returns a state of density $p(s,u;t,\cdot)$.
Its expected total work, including all rejected trials, is uniformly
bounded over the coefficient class, starting states and intervals
with $0<t-s\le h_{\mathrm{env}}$. This bound holds conditionally on
the starting state and times.
\end{lemma}
\begin{proof}
With $h=t-s$, the joint density of one proposal and acceptance is
\[
 G_{4\kappa/3,h}(v-u)\frac{p(s,u;t,v)}{E_h(v-u)}
 =\frac{p(s,u;t,v)}{K_{\mathrm{rej}}}.
\]
Each trial therefore succeeds with probability $1/K_{\mathrm{rej}}$,
and its conditional law given success is the required transition.
The Gaussian proposal and Lemma~\ref{core:shortcoin} give a uniform
bound $c_F$ on the mean work of one trial. Independent restarts and
Lemma~\ref{supp:fw:retry} bound the total expected work by
$c_0+K_{\mathrm{rej}}c_F$, where $c_0$ accounts for preliminary
operations.
\end{proof}

\section{Details of the bridge construction}\label{app:bridge-proofs}

We first quantify the cost of terminal rejection, which motivates the
bridge proposal in Section~\ref{poly:section}. We then prove
Theorem~\ref{bridge:general} using the supporting results below.
These give the local density coins, weighted-path sampler, parameter
choices, component acceptance tests and insertion bounds.

\subsection{The cost of terminal rejection}\label{br5:terminal-rejection}

We compute the acceptance probability and exhibit constant coefficients
for which terminal rejection has exponentially large expected work.
Fix $T>0$, endpoints $x,y$, and an integer
$M\ge\max\{2,\lceil T/h_{\mathrm{env}}\rceil\}$, and let
$\Delta=T/M$. In each trial, use the exact short-transition sampler
to generate an unconditional diffusion skeleton from $x$ at
$\Delta,\ldots,(M-1)\Delta$. Given its last state $u$, accept
the prefix with probability
\begin{equation}\label{br5:terminal-rejection-test}
 \frac{p(T-\Delta,u;T,y)}{U_\Delta},\qquad
 U_\Delta=A_0(2\pi\kappa\Delta)^{-d/2},
\end{equation}
restarting the complete prefix upon rejection. Let $\nu_{\mathrm{tr}}$
be the number of trials and $\Wk_{\mathrm{tr}}$ their total work.

\begin{proposition}[Cost of terminal rejection]
\label{br5:terminal-rejection-lower}
The preceding acceptance rule is implementable with a transition coin
and a known uniform comparison. The accepted prefix has the bridge
law on its grid, and $\E\nu_{\mathrm{tr}}=\frac{U_\Delta}{p(0,x;T,y)}.$
There are constants $c_1,c_2>0$, depending only on the coefficient
class and primitive costs, such that for every $T>0$, every $x$ and
every admissible $M$, this particular construction with $y=x$ satisfies
\begin{equation}\label{br5:terminal-rejection-work}
 \sup_{(b,\sigma)\in\mathcal F_T}\E\Wk_{\mathrm{tr}}
 \ge c_1e^{c_2T}.
\end{equation}
\end{proposition}
\begin{proof}
The Gaussian envelope implies that
\[
 p(T-\Delta,u;T,y)\le E_\Delta(y-u)
 =U_\Delta\exp\{-3|y-u|^2/(8\kappa\Delta)\}\le U_\Delta.
\]
Thus \eqref{br5:terminal-rejection-test} is implemented by a
transition coin and an independent comparison with the displayed
exponential. Multiplication by the unconditional prefix density gives the
unnormalized bridge density divided by $U_\Delta$. By the
Chapman--Kolmogorov identity, its mass is $p(0,x;T,y)/U_\Delta$.
Conditioning on acceptance gives the bridge law, and independent
restarts give the stated mean trial count.

Since the coefficient class is nonempty, choose an admissible
diffusion coefficient $\sigma_*$ and fix a point $(t_*,z_*)$ in its
domain. Set $\Sigma_0=\sigma_*(t_*,z_*)$. This constant matrix
inherits the componentwise and ellipticity bounds of $\sigma_*$.
Choose the constant coefficients
\[
 b(t,z)\equiv\mu:=\frac{\Mc}{2}e_1,\qquad
 \sigma(t,z)\equiv\Sigma_0,\qquad
 A:=\Sigma_0\Sigma_0^\top,
\]
where $e_1$ is the first coordinate vector. These coefficients
satisfy the original assumptions, with all positive-order
derivatives equal to zero. For equal endpoints, their transition
density is
\[
 p(0,x;T,x)=(2\pi T)^{-d/2}(\det A)^{-1/2}
             \exp\{-T\mu^\top A^{-1}\mu/2\}.
\]
Consequently,
\[
 \E\nu_{\mathrm{tr}}
 =A_0\sqrt{\det A}\,\kappa^{-d/2}M^{d/2}
            e^{T\mu^\top A^{-1}\mu/2}
 \ge A_0(2\lambda/\kappa)^{d/2}
            e^{\Mc^2T/(8\Lambda)},
\]
using $M\ge2$, $\det A\ge\lambda^d$ and
$A^{-1}\succeq\Lambda^{-1}\Id$.
Each trial incurs at least a fixed positive primitive cost
$c_{\min}$, so
$\Wk_{\mathrm{tr}}\ge c_{\min}\nu_{\mathrm{tr}}$ pathwise.
Taking expectations proves the claimed work lower bound.
\end{proof}

\subsection{Proof of the bridge sampling guarantee}
\label{br5:bridge-sampler-proof}

The proof first combines the component and insertion bounds for the
zero-endpoint construction, then verifies that the affine
transformation preserves both the coefficient class and the bridge law.

\begin{proof}[Proof of Theorem~\ref{bridge:general}]
Consider steps~2--6 on $[0,2n]$, with zero endpoints. Let $m'$ be
the number of requested times for this construction, including the
midpoint $n$. The number of retained cells and the maximum number
of good neighbors satisfy
\begin{equation}\label{br5:cell-counts}
 S_{\mathrm{cell}}=O(n^{d(3+4/\alpha)}),\qquad
 K_{\mathrm{nb}}=O(n^{d(1+4/\alpha)}),
\end{equation}
since $R\asymp n$, $N\asymp n^4$, $h\asymp n^{-3}$,
$L\asymp n^{-1}$ and $\ell\asymp n^{-2-4/\alpha}$.
Enumerating the integer offsets of length at most $L/\ell$ at each
cell and time layer constructs the cell graph and its known factors
in $O(NS_{\mathrm{cell}}K_{\mathrm{nb}})$ work. The augmented graph
therefore has
$E=O(NS_{\mathrm{cell}}K_{\mathrm{nb}}+N^2+dN)$ edges.

Lemma~\ref{br5:global-guarantee} combines the graph draw and the
selected component's proposal and acceptance test. It gives an exact
coarse bridge with expected work at most
\[
 C\{EN^3\log(4EN)+N^2+n+NS_{\mathrm{cell}}\}
 =O\{n^Q\log(2n)\}.
\]
Here $N^4S_{\mathrm{cell}}K_{\mathrm{nb}}=O(n^Q)$,
$N^5=O(n^{20})=O(n^Q)$ and $\log(4EN)=O(\log(2n))$.
The $n^{20}$ term is absorbed since $d\ge1$ and $0<\alpha\le1$.
Constructing the graph and storing and updating the sampled
states also require at most $O(n^Q\log(2+n))$ work.
Lemma~\ref{br5:insertion} preserves the bridge law and adds
$O(m'+N)$ expected work. Thus the total work $\Wk_0$ satisfies
\begin{equation}\label{bridge:calibrated-summary}
 \E\Wk_0\le C\{n^Q\log(2n)+m'\}.
\end{equation}
All implementations use only coefficient values and the stated
primitives.

For general inputs, use the fixed constant
\begin{equation}\label{br:scale-constant}
 L_{\mathrm{sc}}=\max\left\{1,\frac{2+\sqrt d}{4},
                              \frac{(\Mc+1)^2}{2\Mc^2}\right\}.
\end{equation}
Write $D_{\mathrm{end}}=|y-x|$ and $H=H_{T,x,y}$.
The inequalities $H\ge T$, $H\ge D_{\mathrm{end}}^2/T$ and
$H\ge2D_{\mathrm{end}}$ imply
\[
 r\le(2L_{\mathrm{sc}})^{-1},\qquad
 D_{\mathrm{end}}/S\le(4L_{\mathrm{sc}})^{-1},\qquad
 D_{\mathrm{end}}/\sqrt{TS}\le(2L_{\mathrm{sc}})^{-1/2}.
\]
For a scalar coefficient component $f$ and $p+|\beta|\le k$,
the chain rule gives
\begin{equation}\label{br:affine-derivatives}
 \left|\partial_s^p\partial_z^\beta
 f(rs,x+(s/S)w+\sqrt r\,z)\right|
 \le\Mc\left(r+\frac{\|w\|_1}{S}\right)^p r^{|\beta|/2}
 \le\Mc,
\end{equation}
because $r+\|w\|_1/S\le(2+\sqrt d)/(4L_{\mathrm{sc}})\le1$.
Positive-order derivatives of $\widetilde b$ have the additional
factor $\sqrt r\le1$, while
$|\widetilde b_i|\le(\Mc+1)/\sqrt{2L_{\mathrm{sc}}}\le\Mc$.
The transformed covariance is the original covariance evaluated
at the transformed point, so the ellipticity bounds are unchanged.
The one-sided derivative conventions at the time endpoints are
also preserved. Together with the coefficient and derivative
bounds established above, these facts show that the transformed
coefficients satisfy Assumption \ref{ass:class} on $[0,S]$.

Brownian scaling gives \eqref{br:affine-coefficients}. With
$\Psi_s(z)=x+(s/S)w+\sqrt r\,z$, the continuous densities satisfy
\[
 \widetilde p(s,z;q,\zeta)
 =r^{d/2}p(rs,\Psi_s(z);rq,\Psi_q(\zeta)),\qquad s<q.
\]
Change of variables yields this identity almost everywhere, and
continuity extends it to every endpoint. If $n$ was added to the
requested grid, first integrate out its state by the
Chapman--Kolmogorov identity. Substitution into the bridge density
for the remaining $m$ times gives the factor $r^{md/2}$, exactly
the Jacobian of \eqref{br:affine-output}. Hence the returned states
have the target bridge distribution.

Each transformed scalar query uses one original query and $O(d)$
arithmetic, with time arguments in $[0,T]$. Parameter and grid
processing cost $O(n+m)$, and $m'\le m+1$. Since
$L_{\mathrm{sc}}H\le n\le(L_{\mathrm{sc}}+1)H$ and $H,m\ge1$,
\eqref{bridge:calibrated-summary} gives
\[
 \E\Wk_{A_B}
 \le C_1\{n^Q\log(2n)+m+1+n\}
 \le C\{H^Q\log(2+H)+m\}.
\]
\end{proof}

\subsection{Local coins and relative density variation}
\label{br5:local-inputs-proof}

The graph construction needs two local facts: a coin for the ratio
$p(e)/(2g(e))$ with bounded expected work, and control of the density
as its endpoints move within their cells. We obtain both from the
following local estimates.
For a positive definite matrix $V$, write
\[
 G_{V,h}(z)=(2\pi h)^{-d/2}(\det V)^{-1/2}
                \exp\{-z^\top V^{-1}z/(2h)\}.
\]
We specify the estimator used by $\operatorname{SharpCoin}_n$ on
\eqref{br5:moderate}. The threshold $D_{\sharp}$ is fixed throughout.
Let $L_a=2d^{5/2}\Mc^2$, $\eta_{\mathrm{sh}}=1/8$ and
$r_{\mathrm{sh}}=\min\{1,\lambda\eta_{\mathrm{sh}}/(8L_a)\}$. Set
\[
 A=a(t,v),\qquad \eta=\eta_{\mathrm{sh}}/n,\qquad r_n=r_{\mathrm{sh}}/n,
 \qquad V_{\mathrm{in}}=(1+\eta)A,
 \qquad V_{\mathrm{out}}=2\kappa\Id.
\]
We split each parametrix correction according to whether all its
intermediate states remain in the ball of radius $r_n$ about $v$.
The covariance $V_{\mathrm{in}}$ provides tight bounds inside this
ball, while $V_{\mathrm{out}}$ controls its complement. Both integrals
are retained, using separate Gaussian references to estimate them.

For a correction of order $j\ge1$ and reference matrix $V$, draw the
gaps $\Delta_0,\ldots,\Delta_j$ as in \eqref{core:randomgaps}, and set
$q_0=s$ and $q_{i+1}=q_i+\Delta_i$ for $0\le i\le j$, so that
$q_{j+1}=t$. With $z_0=u$ and $z_{j+1}=v$, generate the intermediate
Gaussian bridge states successively: for $1\le i\le j$, draw $z_i$
conditionally on the gaps and $z_{i-1}$ from the Gaussian distribution
with mean and covariance
\[
 z_{i-1}+\frac{q_i-q_{i-1}}{t-q_{i-1}}(v-z_{i-1}),
 \qquad
 \frac{(q_i-q_{i-1})(t-q_i)}{t-q_{i-1}}V.
\]
We split the integral defining the order-$j$ parametrix correction
into two parts: one where all intermediate states lie within
distance $r_n$ of the terminal point $v$, and the other where at
least one intermediate state lies outside this ball. These two
parts correspond to the event $\mathcal I=\left\{\max_{1\le i\le j}|z_i-v|\le r_n\right\}$
and its complement $\mathcal I^c$, respectively.
For either event $\mathcal E$, define
\begin{equation}\label{br5:sharp-correction}
 U_j^{V,\mathcal E}
 =\mathsf Z_j(h)
 \frac{g(s,u;q_1,z_1)}{G_{V,\Delta_0}(z_1-z_0)}
 \prod_{i=1}^j
 \frac{\sqrt{\Delta_i}H(q_i,z_i;q_{i+1},z_{i+1})}
      {G_{V,\Delta_i}(z_{i+1}-z_i)}\,
 \ind_{\mathcal E}.
\end{equation}
Test $\mathcal E$ before evaluating the product; if it fails, set
this correction to zero. The factors $g,H$ and the recurrence for
$\mathsf Z_j$ are those of the unconditional sampling construction.

Next, independently draw $J_1,J_2$ with
$\Pp(J_i=j)=2^{-j}$, $j\ge1$, by counted fair-coin loops.
Using independent gap and bridge draws for the two parts, generate
\[
 U_{J_1}^{\mathrm{in}}=U_{J_1}^{V_{\mathrm{in}},\mathcal I},
 \qquad
 U_{J_2}^{\mathrm{out}}=U_{J_2}^{V_{\mathrm{out}},\mathcal I^c},
\]
where each event is evaluated on its own sampled bridge. Set
\begin{equation}\label{br5:sharp-estimator}
 Y_n(e)=1+
 2^{J_1}U_{J_1}^{\mathrm{in}}
       \frac{G_{V_{\mathrm{in}},h}(v-u)}{g(e)}
 +2^{J_2}U_{J_2}^{\mathrm{out}}
       \frac{G_{V_{\mathrm{out}},h}(v-u)}{g(e)}.
\end{equation}
The factors $2^{J_1}$ and $2^{J_2}$ compensate for the probabilities
of selecting the correction orders.

To generate $\operatorname{SharpCoin}_n(e)$, construct $Y_n(e)$ as
above, draw an independent $U\sim\operatorname{Uniform}(0,1)$, and
return $\ind\{U\le Y_n(e)/2\}$.
The following lemma verifies the bounds and unbiasedness needed for
this comparison and gives a relative density estimate for the later
spatial discretization.

\begin{lemma}[Local inputs]\label{br5:local-inputs}
For each fixed $D_{\sharp}>0$, there are computable constants $h_0,c_0>0$,
$C_{\mathrm{Hol}}<\infty$ and a positive dyadic exponent
$\alpha\in(0,1]$, depending only on $D_{\sharp}$ and the coefficient bounds,
with the following properties. For every integer $n\ge1$ and every
edge satisfying \eqref{br5:moderate},
\[
 \E Y_n(e)=\frac{p(e)}{g(e)},\qquad
 \frac12\le Y_n(e)\le\frac32\quad\hbox{almost surely}.
\]
Consequently, $\operatorname{SharpCoin}_n(e)$ has mean
$p(e)/(2g(e))\in[1/4,3/4]$ and uniformly bounded expected work.
With $\rho=c_0\sqrt h/(1+\sqrt n)$ and
$|u'-u|,|v'-v|\le\rho/4$, we have
\begin{equation}\label{br5:relative-holder}
 \left|\frac{p(s,u';t,v')}{p(s,u;t,v)}-1\right|
 \le C_{\mathrm{Hol}}
 \left\{\left(\frac{|u'-u|}{\rho}\right)^\alpha
       +\left(\frac{|v'-v|}{\rho}\right)^\alpha\right\}.
\end{equation}
The exponent $\alpha$ may be fixed from the coefficient bounds alone,
independently of $D_{\sharp}$.
\end{lemma}

\begin{proof}[Proof of Lemma~\ref{br5:local-inputs}]
We first bound the estimator in \eqref{br5:sharp-estimator}.
Decrease $h_0$ so that
\[
 h_0\le\min\{1,h_{\mathrm{env}},r_{\mathrm{sh}},r_{\mathrm{sh}}^2/[16(1+D_{\sharp})]\}.
\]
The moderate-displacement condition implies $|u-v|\le r_n/4$.
The known Lipschitz bound
$\|a(q,z)-a(t,v)\|_F\le L_a(|t-q|+|z-v|)$ and the choices of
$r_{\mathrm{sh}},\eta_{\mathrm{sh}}$ then imply
\begin{equation*}
 (1-\eta/4)A\preceq a(q,z)\preceq(1+\eta/4)A,
 \qquad q\in[s,t],\quad |z-v|\le r_n.
\end{equation*}

For subedges of duration $\delta$ with both endpoints in this ball,
the explicit formulas in parametrix expansion imply that
\begin{equation*}
 g\le C_1G_{V_{\mathrm{in}},\delta},\qquad
 |H|\le C_2\eta^{-3/2}\delta^{-1/2}
              G_{V_{\mathrm{in}},\delta},
\end{equation*}
where $C_1,C_2$ are computable class constants. 
For the whole input edge,
\begin{equation*}
 \frac{G_{V_{\mathrm{in}},h}(v-u)}{g(e)}
 \le\exp\{\eta D_{\sharp}n/(2\lambda)\}
 \le C_D:=\exp\{\eta_{\mathrm{sh}}D_{\sharp}/(2\lambda)\}.
\end{equation*}

Write $K_j=g\otimes H^{\otimes j}$ and split its integral into
$K_j^{\mathrm{in}}$ and $K_j^{\mathrm{out}}$ according to
$\mathcal I$ and $\mathcal I^c$. Cancellation of the gap and Gaussian
bridge densities, exactly as in \eqref{core:Tnproperties}, gives
\[
 \E U_j^{\mathrm{in}}
   =\frac{K_j^{\mathrm{in}}}{G_{V_{\mathrm{in}},h}(v-u)},
 \qquad
 \E U_j^{\mathrm{out}}
   =\frac{K_j^{\mathrm{out}}}{G_{V_{\mathrm{out}},h}(v-u)}.
\]
Moreover,
\[
 |U_j^{\mathrm{in}}|
 \le C_1(C_2\eta^{-3/2})^j\mathsf Z_j(h).
\]
For an outside path, some node $z_l$ satisfies $|z_l-v|>r_n$. Hence
\[
 \sum_{i=0}^j\frac{|z_{i+1}-z_i|^2}{\Delta_i}
 \ge\frac{(|z_l-u|+|v-z_l|)^2}{h}
 \ge\frac{(7r_n/4)^2}{h}.
\]
Combining this with Lemma~\ref{core:parametrix} and
$G_{\kappa,\delta}/G_{2\kappa,\delta}
=2^{d/2}\exp\{-|z|^2/(4\kappa\delta)\}$ yields
\begin{equation*}
 |U_j^{\mathrm{out}}|
 \le C_g2^{d/2}(2^{d/2}C_H)^j\mathsf Z_j(h)
          \exp\{-c_{\mathrm{out}}r_n^2/h\},
 \qquad c_{\mathrm{out}}:=49/(64\kappa).
\end{equation*}

The absolute convolution bounds from the parametrix expansion justify
summing these expectations, so \eqref{br5:sharp-estimator} is unbiased
for $p(e)/g(e)$.
Choose $h_0$ small enough that
$2C_2\eta_{\mathrm{sh}}^{-3/2}\sqrt{\pi h_0}\le1/2$ and
$2^{1+d/2}C_H\sqrt{\pi h_0}\le1/2$.
Since $\Gamma(1+j/2)\ge1/2$, then for every $j\ge1$ the absolute inside
correction in \eqref{br5:sharp-estimator} is at most
$4C_1C_D C_2\eta_{\mathrm{sh}}^{-3/2}\sqrt{\pi h_0}$,
and the outside correction has absolute value at most
\[
 2^{d+2}C_gC_H\sqrt{\pi h}
 \exp\left\{\left(\frac{D_{\sharp}}{2\lambda}
                   -\frac{c_{\mathrm{out}}r_{\mathrm{sh}}^2}{h_0}\right)n\right\}.
\]
Decreasing $h_0$ further makes both bounds at most $1/4$, uniformly
for $n\ge1$. 
Thus $1/2\le Y_n(e)\le3/2$ and the uniform comparison has the
claimed mean. Gap generation, Gaussian bridge simulation, ball tests,
coefficient queries and evaluation of the normalizer recurrence cost
$O(j+1)$ at order $j$. Since $\E J_1=\E J_2=2$, the total expected
work is uniformly bounded. 

We prove \eqref{br5:relative-holder} with the dyadic choice
$\alpha=1/2$. In terminal variables,
\begin{equation*}
 \partial_t p=\operatorname{div}
 \left[\tfrac12a\nabla p+
       \{\tfrac12\operatorname{div}a-b\}p\right],
\end{equation*}
while the reversed backward operator is
$\operatorname{div}(\tfrac12a\nabla)
+(b-\tfrac12\operatorname{div}a)\cdot\nabla$.
The principal matrix is Lipschitz and the lower-order coefficients
are bounded, with known class bounds.

These bounds give an effective local spatial $1/2$-H\"older estimate.
Indeed, write $Q_r(q,z)=(q-r^2,q]\times B_r(z)$.
For a solution $f$ bounded by $M$ on $Q_{4R}$, $R\le1$, define
$\mathcal E(r)=\inf_c\int_{Q_r}|f-c|^2$ for centers in $Q_R$.
Then by Caccioppoli's inequality on $Q_{2r}$ and the comparison with the
constant-coefficient heat equation on $Q_r$, we have
\[
 \mathcal E(\theta r)
 \le C_1\theta^{d+4}\mathcal E(r)+C_2M^2r^{d+4},
 \qquad r\le R,\quad 0<\theta\le\tfrac14,
\]
uniformly over these centers. The constants are computable from
the energy inequalities and derivatives of the explicit Gaussian
kernel. Choose a dyadic $\theta$ with $C_1\theta\le1/2$.
Iteration gives $\mathcal E(r)\le CM^2r^{d+3}/R$; telescoping
cylinder averages then yields
$|f(q,z)-f(q,z')|\le CM(|z-z'|/R)^{1/2}$
on the smaller cylinder, including its upper time face.

Now take $c_0=1/8$ and $\rho=c_0\sqrt h/(1+\sqrt n)$.
Use endpoint cylinders of radius $4\rho$ and depth $16\rho^2$,
ending at $t$ and at $s$ in reversed time, respectively.
Their transition durations lie in $[h/2,h]$.
Decrease $h_0$ so that the preceding sharp comparison also holds
with displacement threshold $D_*=4(\sqrt{D_{\sharp}}+5)^2$.
For every edge $e'$ in these cylinders, including a displacement
of the opposite endpoint by at most $\rho/4$, the Gaussian formula
and the Lipschitz bound on $a$ give
\[
 \left|\log\frac{g(e')}{g(e)}\right|
 \le C\left\{\sqrt{n/h}\,\rho+n\rho^2/h+n(\rho+\rho^2)\right\}
 \le C.
\]
Thus $\tfrac12g\le p\le\tfrac32g$ bounds the density on both
cylinders by $Cp(e)$. Applying the local estimate successively
to the two endpoints and dividing by $p(e)>0$ proves
\eqref{br5:relative-holder}. All constants are computable from
the stated bounds. 

For $k=1$, the same conclusion follows by smooth approximation.
Extend $a$ and $b$ beyond the time interval while preserving
uniform ellipticity and controlled first-order bounds, and
mollify the extensions to obtain smooth coefficients
$a_\varepsilon,b_\varepsilon$. These approximations converge
locally uniformly to $a,b$, and the estimates above hold with
constants independent of $\varepsilon$. The corresponding
parametrix kernels converge term by term and satisfy uniform
Gaussian bounds, which provide summable, integrable majorants.
Dominated convergence therefore yields convergence of the
transition densities $p_\varepsilon$ to $p$ for $s<t$.
Passing to the limit in the density comparisons and the relative
Hölder estimate proves the result for the original coefficients.
\end{proof}

\subsection{Implementation of the weighted-path procedure}
\label{br5:weighted-proof}

We now supply the coin transformation and sample sizes used by the
weighted-path procedure in Section~\ref{br5:weighted}, and prove its
exactness and work bound. The same result applies to the cell graph
and to the augmented graph used for component selection. Throughout,
$L_0=2\log16$, $\lambda_0=2NL_0$ and
$c_{\mathfrak e}=1+\log(w_{\mathfrak e})/L_0$.

We use Huber's linear Bernoulli factory
\citep[Theorem~1 and Section~2.2]{Huber2016}. Given coins of mean
$q$, known $a>1$ and $\epsilon\in(0,1)$ with $aq\le1-\epsilon$,
write $\operatorname{LF}(a,\epsilon;q)$ for its output coin of
mean $aq$. When $a$ stays in a fixed compact subset of $(1,\infty)$
and an input coin has uniformly bounded conditional expected cost
$c$, its expected scalar work is $O((c+1)/\epsilon)$.

To construct the $c_{\mathfrak e}$-coins in Section~\ref{br5:weighted}, let
$\rho_0=7/8$ and $C_0=\log8$. Repeatedly propose an integer $J\ge1$
with geometric mass $(1-\rho_0)\rho_0^{j-1}$ and accept with
probability $1/J$. Given the accepted $J$, let $D_{\mathfrak e}$ be the product
of $J$ independent $\operatorname{LF}(8/7,1/7;1-w_{\mathfrak e})$ coins.
With probability $C_0/L_0$ return $1-D_{\mathfrak e}$; otherwise return one.
Complementing an edge coin supplies the input mean $1-w_{\mathfrak e}$.
Geometric proposals are generated by successive uniform comparisons.

The empirical sample sizes used in the main-text procedure are
\[
 J_0=\left\lceil\tfrac12N^2L_0^2
                  \{\lambda_0+\log(2E)\}\right\rceil,\qquad
 B_0=\left\lceil\tfrac12\lambda_0^2\log(4E\lambda_0)\right\rceil.
\]
Recall that for each edge, $\widehat c_{\mathfrak e}$ is the mean of $J_0$ independent
$c_{\mathfrak e}$-coins. These initial estimates remain fixed through the internal
retries of one graph draw. For every Poisson factor,
$\overline c_{\mathfrak e}$ is the mean of $B_0$ fresh such coins at each edge. These fresh estimates are
independent across edges and factors, conditional on the initial
estimates. For the analysis, we define
\[
 \widehat v_\gamma=\tfrac12+\frac1{2N}\sum_{\mathfrak e\in\gamma}d_{\mathfrak e},
 \qquad \widehat v_*=\tfrac12+\frac{M(o)}{2N}.
\]
Both values lie in $[0,1]$ and the maximum is at least the candidate
value, so \eqref{br5:graph-factor} is a valid probability. Then we have the following guarantee.

\begin{lemma}[Sampling a weighted graph path]\label{br5:dag}
For the graph of Section~\ref{br5:weighted}, suppose that the edge
coins have uniformly bounded conditional expected work. The specified
procedure terminates almost surely and samples a path with probability
proportional to $\prod_{\mathfrak e\in\gamma}k_{\mathfrak e}w_{\mathfrak e}$. Its expected scalar work is
$O\{EN^3\log(4EN)\}$, uniformly in the positive factors and the total
path weight.
\end{lemma}
\begin{proof}[Proof of Lemma~\ref{br5:dag}]
The accepted geometric proposal defining $J$ has mass
$\rho_0^j/(jC_0)$. The multiplier $8/7$ is admissible because
$(1-w_{\mathfrak e})/\rho_0\le6/7$, and
\[
 \E D_{\mathfrak e}=\sum_{j\ge1}\frac{\rho_0^j}{jC_0}
       \left(\frac{1-w_{\mathfrak e}}{\rho_0}\right)^j
       =-\frac{\log w_{\mathfrak e}}{C_0}.
\]
Hence the stated mixture returns a $c_{\mathfrak e}$-coin. Both the geometric
rejection loop and its accepted count have bounded means, proving
the uniform work bound for this transformation. The factory's scalar
work assertion includes its geometric inner loops: when its multiplier
belongs to a fixed compact subset of $(1,\infty)$, their means remain
uniformly bounded. Each outer update follows a nonempty inner loop
and uses fixed scalar work, including powers evaluated by exponentials
and logarithms.

Condition on the initial estimates $(\widehat c_{\mathfrak e})_{\mathfrak e}$ and denote
\[
 R_\gamma=L_0\sum_{\mathfrak e\in\gamma}(c_{\mathfrak e}-\widehat c_{\mathfrak e}),\qquad
 v_\gamma=\frac12+R_\gamma/\lambda_0,\qquad
 v_* =\max_\gamma v_\gamma.
\]
The proposal law is proportional to
$\prod_{\mathfrak e\in\gamma}k_{\mathfrak e}\exp\{L_0(\widehat c_{\mathfrak e}-1)\}$.
For fresh edge estimates $(\overline c_{\mathfrak e})_{\mathfrak e}$, let
$z_* =\E(\widehat v_*\mid\text{initial estimates})$.
It is common to all proposed paths, and a factor coin has mean
$1+v_\gamma-z_*$.
Jensen's inequality gives $z_*\ge v_*$.
Hoeffding's inequality and the choice of $B_0$ imply
\[
 \Pp\!\left(\max_{\mathfrak e}|\overline c_{\mathfrak e}-c_{\mathfrak e}|>1/\lambda_0
                 \mid\text{initial estimates}\right)
 \le (2\lambda_0)^{-1}.
\]
On its complement every empirical path error is at most
$1/(2\lambda_0)$, and errors are always at most one. Consequently
$0\le z_*-v_*\le1/\lambda_0$.
The factors use independent sets of edge estimates and fresh uniform
variables. Averaging their product over the Poisson count gives
\[
 A_\gamma
 =\exp\{\lambda_0(v_\gamma-z_*)\}
 =\exp\{\lambda_0(1/2-z_*)\}\exp(R_\gamma).
\]
The first factor is common to all candidate paths. Multiplying this
acceptance probability by the proposal weight therefore gives a
common multiple of $\prod_{\mathfrak e\in\gamma}k_{\mathfrak e}w_{\mathfrak e}$. This proves the exact
path law for every realization of the initial estimates.

To bound retries, let
$\mathcal G=\{\max_{\mathfrak e}|\widehat c_{\mathfrak e}-c_{\mathfrak e}|\le(NL_0)^{-1}\}$.
The choice of $J_0$ gives
$\Pp(\mathcal G^c)\le e^{-\lambda_0}$.
On $\mathcal G$, all $R_\gamma\in[-1,1]$; their range is always
at most $\lambda_0$. If $a_{\rm init}$ is the mean acceptance
conditional on the initial estimates, then we have
\[
 a_{\rm init}\ge e^{-\lambda_0-1},\qquad
 a_{\rm init}\ge e^{-3}\ \text{on }\mathcal G,
 \qquad \E(a_{\rm init}^{-1})\le e^3+e.
\]
Each empirical factor costs $O(EB_0)$ expected work uniformly in
the initial estimates and candidate path. The Poisson loop uses
mean $\lambda_0+1$ exponential increments, including the overshoot,
so one trial costs $O\{EN^3\log(4EN)\}$.
Generating the initial estimates costs
$O\{EN^2(N+\log E)\}$ and is of no larger order.
Conditional on the initial estimates, the trial records are
independent and identically distributed.
Lemma~\ref{supp:fw:retry}, followed by averaging the bound on
$a_{\rm init}^{-1}$, finishes the proof.
\end{proof}

\subsection{Density bounds}
\label{br5:density-proof}

The lower bound below determines how small the graph's omitted
components must be relative to the bridge normalizing constant.
The upper bound is used when averaging the insertion cost under
the bridge law.

\begin{lemma}[Density bounds for the bridge construction]
\label{br5:density-bounds}
There are computable class constants $c_*,C_*,C_{\mathrm{den}}>0$
and $\delta_*>0$ such that
\begin{equation}\label{br5:global-density}
 \begin{gathered}
 p_{T_0}=p(0,0;2n,0)\ge c_*e^{-C_*n},\qquad n\ge1,\\
 \sup_{u,v}p(s,u;t,v)\le C_{\mathrm{den}}\delta^{-d/2}
 \quad\text{if }t-s\ge\delta,\quad0<\delta\le\delta_*.
 \end{gathered}
\end{equation}
\end{lemma}
\begin{proof}
For the lower bound, choose a fixed
$0<\delta\le\min\{h(1),h_{\mathrm{env}},1/2\}$ from
Lemma~\ref{core:localpositive}, and let
\[
 a_*=\tfrac34(2\pi\Lambda\delta)^{-d/2}e^{-1/(16\lambda)}.
\]
Choose a radius $r_*>0$ so that
$r_*\le\sqrt\delta/8$ and the cube
$K=[-r_*/\sqrt d,r_*/\sqrt d]^d$ has volume $v_*\le1/(2a_*)$.
For endpoints
in $K$ and $\delta/2\le\tau\le\delta$, we have
$|v-u|\le2r_*$ and $|v-u|^2/\tau\le1/8$.
Lemma~\ref{core:localpositive} therefore gives
$p\ge3g/4\ge a_*$. Partition
$T_0=2n$ into $M=\lceil T_0/\delta\rceil$ equal durations. They
belong to $[\delta/2,\delta]$, and Chapman--Kolmogorov identity restricted to
$K$ at the internal times implies
\[
 p_{T_0}\ge a_*^M v_*^{M-1}
   =a_*(a_*v_*)^{M-1}\ge c_*e^{-C_*n}.
\]
For example, $c_*=a_*$ and
$C_*=(4/\delta)|\log(a_*v_*)|$ suffice.
For the upper bound, fix $0<\delta'\le h_{\mathrm{env}}$ and a
transition of duration at least $\delta'$. Split off a final interval
of length $\delta'$, use $p\le K_{\mathrm b}G_{c,\delta'}$ there,
and integrate the preceding transition probability. If the whole
duration equals $\delta'$, apply the short-time bound directly.
Thus one may take
$\delta_*=h_{\mathrm{env}}$ and
$C_{\mathrm{den}}=K_{\mathrm b}(2\pi c)^{-d/2}$.

\end{proof}

\subsection{Choice of the fixed parameters}
\label{br5:parameter-choice}

We calibrate all constants before sampling.  The constant $C_*$ used below comes from the density lower bound
$p_{T_0}\ge c_*e^{-C_*n}$ established in
Section~\ref{br5:density-proof}. The density bounds fix
the thresholds for the graph's omitted components; the same
calibration also covers the sharp-coin and short-transition calls
needed for insertion. Set
\begin{equation}\label{br5:insertion-constants}
 \begin{gathered}
 D=\max\{1,(64c/9)(C_*+1)\},\qquad D_0=4c(C_*+1),\\
 D_1=\max\{D_0+1,2cD_0/\lambda+1\},\qquad
 D_{\sharp}=4\max\{D,D_1,1\},\\
 C_g'=\left(\frac{\Lambda+1}{\lambda}\right)^{d/2},\qquad
 V_+=\Mc+\Lambda/2,\qquad R_0=2V_++C_*+1.
 \end{gathered}
\end{equation}
Here $D$ controls bad edges of the graph, $R_0$ determines the box,
and $D_0,D_1,C_g'$ are used for insertion. In particular,
$9D/(64c)\ge C_*+1$ and $R_0-2V_+-C_*=1$; these inequalities
will bound the masses of the proposals outside the graph.

Apply Lemma~\ref{br5:local-inputs} with threshold $D_{\sharp}$,
taking its constant $c_0\le1$. Let $h_{\mathrm{tilt}}>0$ be the
short-transition duration bound for the fixed coefficient class with
bound $\Mc+d2^k\Mc^2$ and the original ellipticity bounds.
This enlarged class contains the tilted diffusions used for exits.
Decrease the local lemma's constant $h_0$ until
\begin{equation*}
 h_0\le\min\{1/2,h_{\mathrm{env}},h_{\mathrm{tilt}}\},
 \qquad L_a\{h_0+\sqrt{D_1h_0}\}\le1/2.
\end{equation*}
The first condition makes all short-transition calls admissible;
the second controls the change of covariance during insertion.
Such a positive $h_0$ exists because the left side of the second
inequality tends to zero with $h_0$.

Finally, the cell width must control both the relative variation of
$p$ and that of $g$. Put
\[
 A_1=2d^2\Mc^2,\qquad
 C_G=\frac{2\sqrt{D_{\sharp}}}{\lambda}
       +\sqrt d\,A_1\left(\frac{\sqrt d}{2\lambda}
                          +\frac{D_{\sharp}}{2\lambda^2}\right),
\]
and choose
\begin{equation}\label{br5:cell-choice}
 0<c_1\le\min\left\{1,\frac1{2\sqrt d},
 \frac{\sqrt D}{2c_0\sqrt d},
 \frac{2(200C_{\mathrm{Hol}})^{-1/\alpha}}{\sqrt d},
 \frac1{300\sqrt d\,C_Gc_0}\right\}.
\end{equation}
All choices depend only on the original class and have fixed finite
computational cost. The exponent $\alpha$ is supplied by Lemma~\ref{br5:local-inputs}
for the original coefficient class.
The following lemma verifies that the parameter choices above give
the within-cell density comparisons and sharp-coin admissibility
required for the cell proposal and its acceptance tests.

\begin{lemma}[Comparison within cells]\label{br5:cell-comparison}
The preceding calibration implies \eqref{br5:cell-errors} and
\eqref{br5:cell-slack}. All sharp-coin calls on good cell edges,
including actual endpoints, satisfy \eqref{br5:moderate}.
\end{lemma}
\begin{proof}
An endpoint differs from its center by at most $\sqrt d\ell/2$,
which is at most $\rho/4$ by \eqref{br5:cell-choice}. The relative
H\"older estimate therefore bounds the density error by
\[
 2C_{\mathrm{Hol}}(\sqrt d c_1/2)^\alpha N^{-1}
 \le(100N)^{-1}.
\]
Also $\sqrt d\ell/L\le\sqrt d c_1c_0/(2\sqrt D)\le1/4$.
Hence both the center edge and every actual edge in its cells have
displacement at most $5L/4$, within the calibrated threshold
$D_{\sharp}$.

For the frozen Gaussian, differentiate its log density along the
segment joining the center endpoints to the actual endpoints. Every
intermediate edge remains within the calibrated displacement threshold.
The covariance derivative bound and ellipticity bound the derivative
along this segment by $C_G\sqrt{n/h}$ times the maximum endpoint
displacement, with $C_G$ as chosen above. Thus
\[
 \left|\log\frac{g(e)}{g(e^0)}\right|
 \le \frac{\sqrt d\ell}{2}C_G\sqrt{n/h}
 \le\frac{\sqrt d c_1c_0C_G}{2N}
 \le\frac1{600N}.
\]
Here $h\le n^{-3}$ controls covariance-derivative terms, and
$N^{-1/\alpha}\le N^{-1}$. Exponentiating yields the second
inequality in \eqref{br5:cell-errors}. Finally, for the graph edge
$\mathfrak e$ corresponding to the actual transition $e$,
\[
 \frac{1+1/(10N)}{1+1/(100N)}
 \le\frac{U_{\mathfrak e}}{p(e)}
 \le\frac{1+1/(10N)}{1-1/(100N)}.
\]
The bounds \eqref{br5:cell-slack} follow by direct subtraction of one.
\end{proof}

\subsection{Acceptance tests inside the selected cells}
\label{br5:cell-acceptance-proof}

The graph draw followed by uniform sampling in its cells has density
$Z_0^{-1}\prod_{\mathfrak e\in\gamma} U_{\mathfrak e}$ on $E_0$. To implement rejection sampling from
this envelope, we need a coin of mean $p(e)/U_{\mathfrak e}$ for each actual
transition edge $e$ represented by a graph edge $\mathfrak e$. Recall that
\[
 q=\frac{p(e)}{2g(e)},\qquad
 q_0=\frac{p(e^0)}{2g(e^0)},\qquad
 \beta=\frac{g(e)}{(1+1/(10N))g(e^0)},\qquad A=\beta q.
\]
Then $A/q_0=p(e)/U_{\mathfrak e}$.

We implement the Bernoulli race in Section~\ref{br5:cell-acceptance}
as follows. For the $A$-coin, apply
$\operatorname{LF}(2\beta,1/4;q/2)$ to the product of a sharp
$q$-coin and an independent fair coin. For the second coin, choose
fairly between a complemented $q_0$-coin and a fresh $A$-coin,
apply $\operatorname{LF}(2,1/(50N);(1-q_0+A)/2)$, and complement
the output. The resulting mean is $q_0-A$. Choose one of the labels
$A$ and $q_0-A$ with probability $1/2$ and toss its coin. On failure,
repeat this choice and toss with fresh randomness; on success, return
one exactly when the label was $A$. All sharp coins, including those
at the cell centers after path selection, are fresh.

\begin{lemma}[Acceptance within cells]\label{br5:ratio}
The race returns a $p(e)/U_{\mathfrak e}$-coin with expected work $O(N)$,
uniformly over good edges and endpoints in their cells. Applying
these coins independently conditional on the states to all $N$
edges gives one cell-component trial with accepted density
$F\ind_{E_0}/Z_0$. The uniform state draws and acceptance tests
have total expected work $O(N^2)$, uniformly in the selected path.
\end{lemma}
\begin{proof}[Proof of Lemma~\ref{br5:ratio}]
The cell errors \eqref{br5:cell-errors} imply
$9/10\le\beta\le1$, so $2\beta$ stays in a fixed compact subset
of $(1,\infty)$. Also $A=\beta q\le q_0\le3/4$, making the first
factory admissible with gap $1/4$. More precisely,
\[
 \frac{A}{q_0}=\frac{p(e)}{U_{\mathfrak e}}
 \le\frac{1+1/(100N)}{1+1/(10N)}
 =\frac{100N+1}{100N+10},
\]
and therefore
\[
 q_0-A\ge\frac14\frac9{100N+10}
 \ge\frac9{440N}>\frac1{50N}.
\]
The multiplier-two factory is thus admissible with its stated gap,
and its complement has mean $q_0-A$ and expected work $O(N)$.
Each race trial succeeds with probability
$(A+q_0-A)/2=q_0/2\ge1/8$. Conditional on success, its label is
$A$ with probability $A/q_0=p(e)/U_{\mathfrak e}$. A trial has expected work
$O(N)$, so Lemma~\ref{supp:fw:retry} gives the same order for the
complete race. This argument permits dependence between a trial's
work and its success bit.

For the cell component, let $e_i$ be the actual transition represented
by graph edge $\mathfrak e_i$ of the selected path. Multiplying its
proposal density by the fresh acceptance probabilities gives
\[
 \frac{\prod_i U_{\mathfrak e_i}}{Z_0}
 \prod_i\frac{p(e_i)}{U_{\mathfrak e_i}}\ind_{E_0}
 =\frac{F\ind_{E_0}}{Z_0}.
\]
For every selected path and every choice of states in its cells,
\eqref{br5:cell-slack} bounds the acceptance probability below by
$(1+C_2/N)^{-N}\ge e^{-C_2}$. Uniform state generation costs
$O(N)$, and the $N$ fresh acceptance coins cost $O(N^2)$ in expectation.
\end{proof}

\subsection{Proposals for paths outside the finite graph}
\label{br5:residual-proof}

For each selected bad-edge or exit label, the construction makes
one proposal and applies its classification and acceptance tests.
The following result verifies the accepted density needed for
component selection in Section~\ref{br5:component-plan}. 

\begin{lemma}[Bad-edge and exit components]
\label{br5:residual-guarantee}
Every comparison in the bad-edge and exit constructions is valid on
its retained event. Their one-trial accepted subdensities are,
respectively,
$F\ind_{B_i}/D_i$ and $F\ind_{X_\zeta}/D_\zeta$. A bad-edge trial
has expected work $O(N+NS_{\mathrm{cell}})$, and an exit trial has
expected work $O(N+n+NS_{\mathrm{cell}})$, uniformly in its label.
\end{lemma}
\begin{proof}
We first verify that the first-violation classification has finite
work. Before a violation occurs, the current state lies in a box
cell, or is the initial state. Enumerate all cells whose coordinate
indices have absolute value at most
$J_{\mathrm{box}}+\lceil L/\ell\rceil+1$.
This finite collection contains every good neighbor of a box cell.
Thus a proposed next state outside this collection creates a bad edge.
Otherwise, a finite scan locates its half-open cell and determines
whether the edge is bad and whether the state lies outside the box.
The bad-edge test is performed first, as in the main construction,
and the scan stops at the first violation.
Writing $S_{\mathrm{cell}}$ for the number of box cells, this collection
contains $O(S_{\mathrm{cell}})$ cells. Classification therefore costs
$O(NS_{\mathrm{cell}})$ work.

For a bad-edge component $B_i$, recall that a bad edge has actual
displacement greater than $3L/4$, since $\sqrt d\ell\le L/4$.
The Gaussian formula and the definition of $\delta_J$ imply
\[
 K_{\mathrm b}G_{c,h}(w)\le\delta_JG_{2c,h}(w)
 \qquad\text{on every bad edge}.
\]
Also, both $K_{\mathrm b}G_{c,h}$ and $G_{2c,h}$ are bounded by
$Q_h$, using $K_{\mathrm b}\ge1$.
These inequalities verify that the known comparison probabilities
in the bad-edge construction lie in $[0,1]$ on the retained event.

For $i<N$, the proposal density before classification is
\[
 G_{2c,h}(w_i)\prod_{\substack{1\le j<N\\j\ne i}}p_j.
\]
Conditional on the proposed states, the transition coin and uniform
comparison on edge $i$ jointly succeed with probability
\[
 \frac{p_i}{K_{\mathrm b}G_{c,h}(w_i)}
 \frac{K_{\mathrm b}G_{c,h}(w_i)}
      {\delta_JG_{2c,h}(w_i)}
 =\frac{p_i}{\delta_JG_{2c,h}(w_i)}.
\]
The two tests on the last edge jointly succeed with probability
$p_N/Q_h$.
Multiplying these probabilities by the proposal density and the
classification indicator gives
$F\ind_{B_i}/(\delta_JQ_h)=F\ind_{B_i}/D_i$.
For $i=N$, the proposal density is $\prod_{j<N}p_j$, and the two
last-edge tests jointly succeed with probability
$p_N/(\delta_JQ_h)$ on $B_N$, giving the same identity.
Each bad-edge proposal trial has expected work
$O(N+NS_{\mathrm{cell}})$.

We next consider an exit component $X_\zeta$.
The proposal uses drift $b+a\theta$ until the first coarse-grid
observation outside $\mathcal B$, and drift $b$ thereafter.
By the product rule, derivatives of $a_{ij}$ of total order
$r\le k$ are bounded by $d2^r\Mc^2$.
This proves the enlarged coefficient bound stated in the main text.
Since the diffusion coefficient is unchanged, the original
ellipticity bounds also hold. The short-transition sampler therefore
applies in each drift regime, using only original coefficient
value queries.

Keep the main-text convention $\tau=\infty$ when no coarse-grid
exit occurs, and define $\bar\tau=\tau\wedge(T_0-h).$
Write $\mathbb Q_\theta$ for the law of the proposal process through
$T_0-h$. The time $\bar\tau$ is a bounded stopping time.
Boundedness of $\sigma^\top\theta$ gives Novikov's condition, so
Girsanov's theorem yields the likelihood of the original unconditional diffusion law through $T_0-h$, relative to $\mathbb Q_\theta$:
\begin{equation*}
 \mathcal L
 =\exp\left\{
   -\theta\cdot X_{\bar\tau}
   +\int_0^{\bar\tau}V(t,X_t)\dd t
 \right\}.
\end{equation*}

On the retained event $X_\zeta$, we have $\bar\tau=\tau$,
$\theta\cdot X_\tau\ge B\ge R_-$ and $V_-\le V\le V_+$.
Consequently, $\mathcal L\le\delta_X$ and
\[
 \frac{\mathcal L}{\delta_X}
 =e^{-\theta\cdot X_\tau+R_-}
  e^{V_+(\tau-T_0)}
  \exp\left\{\int_0^\tau
       (V(t,X_t)-V_+)\dd t\right\}.
\]
The first two factors are implemented by the two one-time uniform
comparisons in \eqref{br5:exit-tests}.
Conditionally on the full proposal path, the probability that all
comparisons at the independent Poisson times before $\tau$ succeed is
\[
 \exp\left\{
 \int_0^\tau q_V
 \left(\frac{V(t,X_t)-V_-}{q_V}-1\right)\dd t
 \right\}
 =
 \exp\left\{\int_0^\tau
       (V(t,X_t)-V_+)\dd t\right\}.
\]
Thus these Poisson comparisons implement the remaining factor.

Simulation at the merged times has exactly the joint law used in
this calculation. To see this, first consider a full proposal
process and an independent Poisson process, then observe the
proposal process at the merged times. The drift switches only at
the first coarse-grid exit, so observing additional states at
Poisson times does not change its law.
The exact short-transition sampler generates these observations
with the required joint distribution.
The comparisons in \eqref{br5:exit-tests} therefore implement
$\mathcal L/\delta_X$ after averaging over the auxiliary observations.
The final transition coin and uniform comparison jointly have mean
$p_N/Q_h$. Applying change of measure and integrating out the
auxiliary observations gives the accepted density
\[
 \frac{F\ind_{X_\zeta}}{\delta_XQ_h}
 =\frac{F\ind_{X_\zeta}}{D_\zeta}.
\]

The expected number of Poisson times is $q_V(T_0-h)=O(n)$.
Their generation includes the overshooting exponential draw, and
merging the two ordered time grids costs $O(N+n)$ in expectation.
Every merged gap is at most $h$, so each transition has uniformly
bounded conditional expected work under the active drift.
Potential evaluations and uniform comparisons cost $O(n)$ in
expectation. Including classification, one exit proposal trial
therefore has expected work
$O(N+n+NS_{\mathrm{cell}})$.

\end{proof}

\Needspace{6\baselineskip}
\subsection{The global rejection trial}
\label{br5:global-trial-proof}

We now combine the component identities. This verifies the mixture
principle of Section~\ref{br5:component-plan}, including selection
through the augmented graph, and bounds the work of complete retries.

\begin{lemma}[The coarse bridge trial]\label{br5:global-guarantee}
One complete trial, consisting of an augmented-graph draw and the
selected component's proposal and acceptance test, has accepted
density $F/M_{\mathrm{env}}$,
where $M_{\mathrm{env}}\le Cp_{T_0}$. Its expected work is at most
\[
 C\{E_{\mathrm{aug}}N^3\log(4E_{\mathrm{aug}}N)
        +N^2+n+NS_{\mathrm{cell}}\},
\]
where $E_{\mathrm{aug}}$ is the number of augmented-graph edges.
Independent complete retries return an exact coarse bridge and
increase this bound by at most a class constant.
\end{lemma}
\begin{proof}
Lemma~\ref{br5:dag} selects an original cell path or an auxiliary
component path proportionally to its weight. Thus it chooses the
cell component with probability $Z_0/M_{\mathrm{env}}$, a bad-edge
component with probability $D_i/M_{\mathrm{env}}$, and an exit
component with probability $D_\zeta/M_{\mathrm{env}}$. Conditional
on selecting the cell component, it also gives the required cell-path
law. The subsequent proposals and acceptance tests use fresh
randomness. Lemmas~\ref{br5:ratio} and~\ref{br5:residual-guarantee}
therefore imply one-trial accepted density
\[
 \frac{Z_0}{M_{\mathrm{env}}}
       \frac{F\ind_{E_0}}{Z_0}
 +\sum_{i=1}^N
       \frac{D_i}{M_{\mathrm{env}}}
       \frac{F\ind_{B_i}}{D_i}
 +\sum_\zeta
       \frac{D_\zeta}{M_{\mathrm{env}}}
       \frac{F\ind_{X_\zeta}}{D_\zeta}
 =\frac{F}{M_{\mathrm{env}}},
\]
because $E_0$, the $B_i$ and the $X_\zeta$ partition skeleton space.
The bad-edge test takes priority at the first violation, so none of
these events overlap.

We next bound the total envelope mass. By the definitions of the
residual masses,
\[
 M_{\mathrm{env}}=Z_0+N\delta_JQ_h+2d\delta_XQ_h.
\]
Let $p_{T_0}=p(0,0;T_0,0)$ and use
$p_{T_0}\ge c_*e^{-C_*n}$.
The choices in \eqref{br5:insertion-constants} give
$9D/(64c)>C_*$ and $R_0>2V_++C_*$, while
\eqref{br5:cell-choice} gives $\ell\le1$ and $R_->0$.
Since $N\asymp n^4$, $h\asymp n^{-3}$ and
$R_-=R_0(n+1)-\ell/2$, we obtain
\[
 \frac{N\delta_JQ_h}{p_{T_0}}
 \le Cn^{4+3d/2}
       e^{-[9D/(64c)-C_*]n}
 \le C,
 \qquad
 \frac{2d\delta_XQ_h}{p_{T_0}}
 \le Cn^{3d/2}
       e^{-[R_0-2V_+-C_*]n}
 \le C.
\]
The within-cell bound also implies
\[
 Z_0
 \le (1+C_2/N)^N\int_{E_0}F
 \le e^{C_2}p_{T_0}.
\]
Therefore $M_{\mathrm{env}}\le Cp_{T_0}$, uniformly in $n$ and
the coefficients within the class.
Each complete trial succeeds with probability
$p_{T_0}/M_{\mathrm{env}}\ge C^{-1}$.
Conditional on success, its density is
\[
 \frac{F/M_{\mathrm{env}}}{p_{T_0}/M_{\mathrm{env}}}
 =\frac{F}{p_{T_0}},
\]
the exact coarse-bridge density.

For completeness, the work of one trial includes the augmented
graph draw as well as the selected component's proposal and tests.
Let $E_{\mathrm{aug}}$ be the number of edges in the augmented graph.
The original edge coins have uniformly bounded expected work,
and the auxiliary edge coins return one deterministically.
Lemma~\ref{br5:dag} therefore bounds the expected work of the graph
draw, including its initial estimates and internal retries, by
$O\{E_{\mathrm{aug}}N^3\log(4E_{\mathrm{aug}}N)\}$.
The within-cell proposal and acceptance tests cost $O(N^2)$ in expectation,
uniformly over the selected original graph path.
The bad-edge and exit trial bounds proved above consequently give
the following uniform bound for one complete trial:
\[
 c_{\mathrm{trial}}
 =C\left\{
 E_{\mathrm{aug}}N^3\log(4E_{\mathrm{aug}}N)
 +N^2+n+NS_{\mathrm{cell}}
 \right\}.
\]

After a failed trial, the algorithm restarts the augmented graph
draw with fresh initial estimates and fresh component randomness.
Thus complete trials are independent and identically distributed
for fixed coefficients and calibrated graph.
Lemma~\ref{supp:fw:retry} bounds the expected work through the first
success by
\[
 c_{\mathrm{trial}}\frac{M_{\mathrm{env}}}{p_{T_0}}
 \le Cc_{\mathrm{trial}}.
\]
The one-time graph construction cost is accounted for separately
in Section~\ref{br5:bridge-sampler-proof}.

\end{proof}

\subsection{Insertion of the requested times}\label{br5:insertion-proof}

It remains to insert the requested observations into the accepted
coarse bridge. We first verify the conditional sampling law and
cost for fixed neighboring states. The cost for large endpoint
displacements is then averaged under their exact bridge distribution;
this averaging is essential for the uniform bound per insertion.

The Gaussian densities $\mathcal E_j/M_j$ used for moderate
endpoint displacement have parameters
\[
 \begin{aligned}
 \mu_0&=(\tau_2u+\tau_1v)/h',
 &\Sigma_0&=\tau_1\tau_2 A_+/h',\\
 \mu_1&=(\tau_2u+2\tau_1v)/(2\tau_1+\tau_2),
 &\Sigma_1&=\frac{2c\tau_1\tau_2}{2\tau_1+\tau_2}\Id,\\
 \mu_2&=(2\tau_2u+\tau_1v)/(\tau_1+2\tau_2),
 &\Sigma_2&=\frac{2c\tau_1\tau_2}{\tau_1+2\tau_2}\Id.
 \end{aligned}
\]
These formulas follow by completing the square in the Gaussian
products; their convolution integrals give the masses $M_j$.

\begin{lemma}[Insertion work]\label{br5:insertion}
Starting from an exact coarse bridge, the procedure of
Section~\ref{br5:insert-section} inserts all requested times with
the exact joint bridge law and terminates almost surely. Its expected
additional work is $O(m+N)$, uniformly in $n$, the coefficient class
and the requested spacings. Each insertion, once its neighboring
states are identified, has uniformly bounded expected sampling work
under their bridge law.
\end{lemma}
\begin{proof}[Proof of Lemma~\ref{br5:insertion}]
Fix an insertion time and condition on the retained states.
Write $p_i=p(e_i)$, and first suppose that
$|v-u|^2/h'\le D_0n$.
On $S_0$, the covariance of the first frozen kernel differs from
$A=a(r,v)$ by at most
\[
 \|a(t,z)-A\|_F
 \le L_a\{h'+\sqrt{D_1nh'}\}
 \le\frac{L_a\{h_0+\sqrt{D_1h_0}\}}{n}
 \le\frac1{2n}.
\]
The second frozen covariance is $A$.
Both are therefore bounded above by $A_+=A+\Id/n$.
Ellipticity and $A_+\preceq(\Lambda+1)\Id$ give
\[
 g(e_1)\le C_g'G_{A_+,\tau_1}(z-u),
 \qquad
 g(e_2)\le C_g'G_{A_+,\tau_2}(v-z).
\]
Moreover, $\tau_i\le h\le h_0n^{-3}$, and the displacement
conditions defining $S_0$ make both sharp-coin calls admissible.
Thus the two sharp coins and the two uniform comparisons have
product mean $p_1p_2/\mathcal E_0(z)$.

On $S_1$,
\[
 \frac{K_{\mathrm b}G_{c,\tau_1}(z-u)}
      {\delta_1G_{2c,\tau_1}(z-u)}
 =
 \exp\left\{\frac{D_1n-|z-u|^2/\tau_1}{4c}\right\}
 \le1.
\]
Combining this comparison with the two transition coins gives
mean $p_1p_2/\mathcal E_1(z)$.
The calculation on $S_2$ is identical for the second subinterval.
Since the sets $S_j$ partition $\R^d$, choosing component $j$
with probability $M_j/(M_0+M_1+M_2)$ and then applying its tests
gives the one-trial accepted density
\[
 \frac{p(s,u;t,z)p(t,z;r,v)}{M_0+M_1+M_2},
\]
which has mass $p(s,u;r,v)/(M_0+M_1+M_2)$, and conditioning on
acceptance gives the required conditional bridge density.

We next bound the expected work of these mixture trials.
Covariance comparison gives
\[
 \frac{G_{A_+,h'}(v-u)}{g(s,u;r,v)}
 \le
 \exp\left\{\frac{|v-u|^2}{2\lambda^2nh'}\right\}
 \le e^{D_0/(2\lambda^2)},
\]
so $M_0\le Cg(s,u;r,v)$.
The Gaussian covariances in $M_1$ and $M_2$ lie between
$ch'\Id$ and $2ch'\Id$. Hence
\[
 M_1+M_2
 \le C\exp\left\{
       \left[-\frac{D_1}{4c}+\frac{D_0}{2\lambda}\right]n
       \right\}
       g(s,u;r,v)
 \le Cg(s,u;r,v).
\]
Here \eqref{br5:insertion-constants} makes the exponent negative.
The existing edge also satisfies the conditions of
Lemma~\ref{br5:local-inputs}, which gives
$p(s,u;r,v)\ge\tfrac12g(s,u;r,v)$.
Thus $M_0+M_1+M_2\le Cp(s,u;r,v)$, so each mixture trial has
a uniformly positive success probability.

The Gaussian draws, matrix operations, uniform comparisons and
coin calls have uniformly bounded expected work in fixed dimension.
Throughout retries, the existing states and insertion time remain
fixed, while the mixture index, candidate state and acceptance
randomness are regenerated independently.
Lemma~\ref{supp:fw:retry} therefore gives uniformly bounded
conditional expected sampling work for moderate endpoints.

For $|v-u|^2/h'>D_0n$, the Gaussian bridge proposal has density
\[
 \frac{G_{c,\tau_1}(z-u)G_{c,\tau_2}(v-z)}{G_{c,h'}(v-u)}.
\]
The two transition coins give accepted density
$p_1p_2/(K_{\mathrm b}^2G_{c,h'}(v-u))$.
This again yields the required conditional bridge law.
Each trial has uniformly bounded conditional expected work, so
Lemma~\ref{supp:fw:retry} bounds the conditional expected sampling
work by
\[
 C\frac{G_{c,h'}(v-u)}{p(s,u;r,v)}.
\]
The fixed preliminary operations can be absorbed into this bound,
since $p(s,u;r,v)\le K_{\mathrm b}G_{c,h'}(v-u)$.
For every fixed pair of endpoints, the acceptance probability is
positive, so this insertion also terminates almost surely.

It remains to average the last cost on the large-displacement
event. Each successful insertion preserves the joint bridge law,
by the conditional-density calculations above.
The times $s,r$ are deterministic functions of the requested
grid and coarse grid; they do not depend on the sampled states.
Writing $T_0=2n$ and $p_{T_0}=p(0,0;T_0,0)$, their bridge marginal
therefore gives the expected large-displacement contribution
\begin{equation}\label{br5:insertion-cancel}
 \frac{C}{p_{T_0}}\iint
 p(0,0;s,u)G_{c,h'}(v-u)
 \ind_{\{|v-u|^2/h'>D_0n\}}
 p(r,v;T_0,0)\,\dd u\,\dd v.
\end{equation}
Here the factor $p(s,u;r,v)$ in the endpoint marginal cancels
the denominator of the conditional work bound.
At $s=0$, the first transition measure is a point mass at zero;
at $r=T_0$, the final endpoint is fixed at zero.

If $r<T_0$, then $r$ is a right coarse endpoint and
$T_0-r\ge h$.
By \eqref{br5:global-density}, the last density in
\eqref{br5:insertion-cancel} is at most $Ch^{-d/2}$, while
$p_{T_0}^{-1}\le Ce^{C_*n}$.
Also,
\[
 \int G_{c,h'}(w)
 \ind_{\{|w|^2/h'>D_0n\}}\,\dd w
 \le Ce^{-D_0n/(4c)}.
\]
Integrating the first transition probability therefore bounds
\eqref{br5:insertion-cancel} by
\[
 Ch^{-d/2}
 \exp\{[C_*-D_0/(4c)]n\}.
\]
This is uniformly bounded because $h\asymp n^{-3}$ and
$D_0/(4c)=C_*+1$.

If $r=T_0$, then $v=0$ and $s\ge T_0-h\ge1$.
The first density has a fixed supremum bound by
\eqref{br5:global-density}.
Integrating over $u$ and using the same Gaussian tail estimate
gives the preceding bound without the factor $h^{-d/2}$.
Thus the large-displacement contribution is uniformly bounded
in this case as well.

Combining the moderate and larger-displacement estimates shows
that each insertion, once its neighboring states are identified,
has uniformly bounded unconditional expected sampling work.
Linearity of expectation gives $O(m)$ for all such sampling work.
Merging the ordered grids, identifying neighboring states, copying
coincident states and writing the outputs requires $O(N+m)$
additional work.
\end{proof}

\section{Proof of the minimax guarantees}\label{app:minimax-proofs}

\begin{proof}[Proof of Theorem~\ref{main:rates}]
Theorems~\ref{fw:upper} and~\ref{bridge:general} establish admissibility
of $A_F$ and $A_B$. Since writing each requested state has positive
cost, every admissible algorithm in either family satisfies
$\Wk_A\ge c_{\mathrm{out}}m$ for a fixed $c_{\mathrm{out}}>0$.
We first establish the additional lower bound in $t_m$ for unconditional diffusion sampling.

Fix $T$, $x$, the grid and an admissible unconditional exact diffusion algorithm $A$,
and let $\tau=t_m$. Nonemptiness of $\mathcal F_T$ supplies a matrix
$\Sigma_0$ satisfying
\[
 |(\Sigma_0)_{ij}|\le\Mc,\qquad
 \lambda\Id\preceq\Sigma_0\Sigma_0^\top\preceq\Lambda\Id:
\]
take the value of an admissible diffusion coefficient at any one point.
The constant coefficient $\sigma(t,z)=\Sigma_0$ satisfies the same
bounds and has all positive-order derivatives zero. Set $v=(\Sigma_0\Sigma_0^\top)_{11}\le\Lambda$.

Suppose first that $\tau\ge1$ and let $N=\lfloor\tau\rfloor$.
Choose a fixed $\phi\in C_c^\infty((0,1))$ with positive integral
$\bar\phi$ and $\|\phi^{(r)}\|_\infty\le1$ for $0\le r\le k$,
extended by zero outside $(0,1)$. Such a function is obtained by
scaling any nonzero nonnegative smooth bump. Put
$\delta=\Mc/2$ and $\mu=\delta\bar\phi>0$. For
$\varepsilon\in\{-1,1\}^N$, define
\begin{equation}\label{fw:sign-family}
 b_\varepsilon(t,z)=\delta e_1
       \sum_{j=0}^{N-1}\varepsilon_j\phi(t-j),\qquad
 \sigma_\varepsilon(t,z)=\Sigma_0.
\end{equation}
Here $e_1$ is the first coordinate vector. The temporal supports are
disjoint, all spatial derivatives vanish, and the drift is zero after
$N$. Thus every pair in this finite family belongs to $\mathcal F_T$.

Give the signs an independent fair prior, independently of the
algorithmic randomness. In the next calculations, expectations include
both sources of randomness. Let $S=\sum_{j=0}^{N-1}\varepsilon_j$,
and let $Y$ be the first coordinate of the algorithm's last output
minus $x_1$. Exactness for each member of the family gives
\begin{equation}\label{fw:sign-moments}
 Y\mid\varepsilon\sim\calN(\mu S,v\tau),\qquad
 \E[YS]=\mu N,\qquad \E Y^2=\mu^2N+v\tau.
\end{equation}

Represent the algorithmic randomness by an independent sequence of
the allowed uniform and Gaussian draws, called the random tape.
Fixing this sequence and the query replies determines the execution.
Only the draws actually used are charged. Since the sign family is
finite and $A$ terminates almost surely on each member, there is a
common probability-one set of tapes on which all these executions
terminate.

For the convenience of the proof, reveal $\varepsilon_j$ whenever the algorithm
first makes a coefficient query at a time in $[j,j+1)$,
$0\le j<N$. This includes queries whose actual replies do not depend
on that sign; the additional information does not change the algorithm.
Let $N_R$ be the number of signs revealed, $S_R$ their sum, and
$\mathcal H$ the sigma-field generated by the complete query record,
these revelations and the full random tape. Each revealed sign
requires a distinct coefficient query, and $Y$ is
$\mathcal H$-measurable. We claim that
\begin{equation}\label{fw:read-identities}
 \E[S\mid\mathcal H]=S_R,\qquad
 \E S_R^2=\E N_R.
\end{equation}
With the tape fixed, changing any unrevealed signs leaves every query
reply and hence the entire execution unchanged. Conditional on the
record and tape, these signs therefore remain independent and fair,
which proves the first identity. For the second, order the first
revelations and pad this sequence with zeros to length $N$.
Given the tape and the preceding revelations, the next previously unqueried
time interval, or termination without another revelation, is determined.
The sign in any such new interval is still independent and fair. The
padded increments consequently have conditional mean zero and are
orthogonal. Their sum is $S_R$ and their sum of squares is $N_R$,
proving the second identity. Spatially adaptive queries provide no
additional information because the coefficients in
\eqref{fw:sign-family} are independent of space.

Using \eqref{fw:sign-moments}, \eqref{fw:read-identities} and
Cauchy--Schwarz inequality, we obtain
\[
 \mu^2N^2=(\E[YS_R])^2
 \le (\mu^2N+v\tau)\E N_R.
\]
Since $\tau/2\le N\le\tau$ and $v\le\Lambda$, then
\begin{equation*}
 \E N_R\ge\frac{\mu^2N^2}{\mu^2N+v\tau}
 \ge\frac{\mu^2}{4(\mu^2+\Lambda)}\,\tau.
\end{equation*}
Each revelation requires a coefficient query with positive cost,
separate from the output writes. The average expected work over the
finite sign family is therefore at least
$c_1\tau+c_{\mathrm{out}}m$, so some deterministic member has at
least this expected work. The constants depend only on the class and
primitive costs. This argument does not require finite expected work
of $A$: its output moments in \eqref{fw:sign-moments} are finite by
exactness, and $N_R\le N$.

For $0<\tau<1$, output cost alone suffices because $\tau+m\le2m$.
Combining the two cases gives, for every admissible $A$ and every requested grid,
\begin{equation}\label{fw:fixed-grid-lower}
 \sup_{(b,\sigma)\in\mathcal F_T}
 \E_{b,\sigma}\Wk_A(T,x,\mathbf t)\ge c(t_m+m).
\end{equation}
Together with Theorem~\ref{fw:upper}, this also proves
fixed-grid optimality. For the grid supremum, choose
$t_i=iT/m$, so $t_m=T$. The lower bound holds for every admissible
algorithm on this grid, while Theorem~\ref{fw:upper} gives
$\E\Wk_{A_F}\le C(T+m)$ uniformly over all allowed grids.
Taking the infimum over algorithms proves part~(i).

For bridges, fix $T,x,y$ and write $H=H_{T,x,y}$. The output cost
gives $\mathfrak R_B^{x,y}(T,m)\ge c_{\mathrm{out}}m$.
Theorem~\ref{bridge:general} gives
\[
 \sup_{\substack{(b,\sigma)\in\mathcal F_T\\
                 \mathbf t\in\mathcal G_B(T,m)}}
 \E_{b,\sigma}\Wk_{A_B}(T,x,y,\mathbf t)
 \le C_B\{H^Q\log(2+H)+m\}.
\]
Since $m\ge1$, the right side is at most
$C_{T,|y-x|}m$, where one may take $C_{T,|y-x|}=C_B\{H^Q\log(2+H)+1\},$
which proves part~(ii).
\end{proof}

\section*{Acknowledgments}
The authors used OpenAI GPT-5.6 and GPT-6 for assistance with proof
exploration, exposition, and editing. J. Blanchet gratefully acknowledges support from the Office of Naval Research under award N00014-24-1-2655 and the National Science Foundation under grants NSF-CCF-2403008 and NSF-CCF-2312204.

\end{document}